\documentclass[12pt]{amsart}

\usepackage[margin=1.15in]{geometry}

\usepackage{amsmath,amscd,amssymb,amsfonts,latexsym,wasysym, mathrsfs, enumitem, mathtools,hhline,color,ulem,mathrsfs}
\usepackage[all, cmtip]{xy}

\usepackage{url}

\definecolor{hot}{RGB}{65,105,225}

\usepackage[pagebackref=true,colorlinks=true, linkcolor=hot ,  citecolor=hot, urlcolor=hot]{hyperref}

\usepackage{graphicx,enumerate}

\newcommand{\CN}{\mathbb{C}^{n}}
\newcommand{\C}{\mathbb{C}}

\newcommand{\Z}{\mathbb{Z}}
\newcommand{\K}{\mathcal{L}}

\newcommand{\V}{\mathcal{V}}
\newcommand{\W}{\mathcal{W}}

\newcommand{\CC}{\mathcal{CC}}
\newcommand{\IH}{\mathrm{IH}}

\theoremstyle{plain}
\newtheorem{theorem}{Theorem}[section]
\newtheorem{prop}[theorem]{Proposition}

\newtheorem{lm}[theorem]{Lemma}

\newtheorem{cor}[theorem]{Corollary}

\newtheorem{lemma}[theorem]{Lemma}
\newtheorem{thrm}[theorem]{Theorem}

\theoremstyle{definition}

\newtheorem{defn}[theorem]{Definition}
\newtheorem{que}[theorem]{Question}
\newtheorem{rmk}[theorem]{Remark}

\newtheorem{ex}[theorem]{Example}
\newtheorem*{ex*}{Example}

\def\be{\begin{equation}}
\def\ee{\end{equation}}

\def\bt{\begin{thrm}}
\def\et{\end{thrm}}

\def\bc{\begin{cor}}
\def\ec{\end{cor}}

\def\br{\begin{rmk}}
\def\er{\end{rmk}}

\def\bp{\begin{prop}}
\def\ep{\end{prop}}

\def\bl{\begin{lm}}
\def\el{\end{lm}}

\def\bex{\begin{ex}}
\def\eex{\end{ex}}

\def\bd{\begin{defn}}
\def\ed{\end{defn}}
\def\sD{\mathscr{D}}

\newcommand\sA{{\mathcal A}}

\newcommand\sG{{\mathcal G}}

\newcommand\sV{{\mathcal V}}
\newcommand\sL{\mathcal{L}}

\newcommand\sB{\mathcal{B}}

\newcommand\pp{{\mathbb{P}}}

\DeclareMathOperator{\codim}{codim}              
                  
\DeclareMathOperator{\ord}{ord}

\DeclareMathOperator{\IC}{IC}

\DeclareMathOperator{\Perv}{Perv}

\def\ra{\rightarrow}

\def\bC{\mathbb{C}}

\def\al{\alpha}

\def\cH{\mathcal{H}}

\def\cO{\mathcal{O}}

\def\bQ{\mathbb{Q}}

\def\bN{\mathbb{N}}

\newcommand{\hooklongrightarrow}{\lhook\joinrel\xrightarrow}

\title{Length  of Perverse Sheaves on Hyperplane Arrangements}

\author{Nero Budur}
\address{Department of Mathematics, KU Leuven, 
Celestijnenlaan 200B, B-3001 Leuven, Belgium} 
\email{nero.budur@kuleuven.be}
\author{Yongqiang Liu}
\address{Basque Center for Applied Mathematics, Alameda de Mazarredo 14, 
48009 Bilbao, Spain} 
\curraddr{The Institute of Geometry and Physics, University of Science and Technology of China,  No.96,  JinZhai Road Baohe District, Hefei, Anhui, 230026, P.R.China} 
\email{yliu@bcamath.org}

\begin{document}


 \begin{abstract} In this article we address the length of perverse sheaves arising as direct images of rank one local systems on complements of hyperplane arrangements. In the case of a cone over an essential line arrangement with at most triple points, we provide  combinatorial formulas for these lengths. As by-products, we also obtain in this case combinatorial formulas for the intersection cohomology Betti numbers of rank one local systems on the complement with same monodromy around the planes.
 \end{abstract}

\maketitle

\begin{center}
{\it Dedicated to the memory of Prof. \c{S}tefan Papadima}
\end{center}

\setcounter{tocdepth}{1}

\tableofcontents

\section{Introduction}
\subsection{Overview.}

Perverse sheaves and intersection cohomology are fundamental objects encoding the complexity of the topology of a stratified space.  Every perverse sheaf on a complex algebraic variety (and using field coefficients for sheaves) has a finite maximal filtration, called {\it composition series}, with non-zero simple successive quotients. The {\it length} of a perverse sheaf counts the number of simple objects in any composition series. There is currently no known algorithm to compute the length of a perverse sheaf. 

The simplest perverse sheaves are the local systems. Understanding the length of local systems amounts to understanding the geometry of the GIT quotient map from the space of representations of the fundamental group of the variety to the moduli of (semi-simple) local systems. 

A next natural class of perverse sheaves to consider are direct images via the open embedding of local systems on the complement of a hypersurface. In this article we address the length of such perverse sheaves on $\CN$ constructed from rank one local systems on the complement of an arrangement of hyperplanes. 

\subsection{Notation.}

We denote by $\Perv(\CN)$ the category of $\bC$-perverse sheaves on $\CN$. Let $G(\CN)$ be the Grothendieck group of $\Perv(\CN)$, namely the free abelian group on symbols $[P]$, one for each   perverse sheaf $ P\in \Perv(\CN)$, 
modulo the subgroup generated by the relation  $[Q]=[P]+[R]$ for every short exact sequence of perverse sheaves $0\to P\to Q\to R\to 0$. The collection of isomorphism classes of simple perverse sheaves gives a basis for  $G(\CN)$. For any $P\in \Perv(\CN)$, $[P]$ has an unique way to be written down as a sum over this basis in $G(\CN)$:
$$ [P]= \sum_{[Q]} a_{[Q]}(P)\cdot[Q],$$
where $[Q]$ runs over the isomorphism classes of simple perverse sheaves $Q$, and $a_{[Q]}(P)$ are non-negative integers of which only finitely many of them are non-zero.  The length of the perverse sheaf $P$ is $$ \ell(P)= \sum_{[Q]} a_{[Q]}(P).$$ 
For $P, R\in \Perv(\CN)$, we say that $$[P]\geq [R]$$ if $a_{[Q]}(P)\geq a_{[Q]}(R)$ for every $[Q]$.

The characteristic cycle of $P$ will be denoted by $\CC(P)$. Recall that $\CC$ factors through the Grothendieck group $G(\CN)$, see \cite[Section 4.3]{D2}. { Note that the Grothendieck group of constructible sheaves is isomorphic to the Grothendieck group of perverse sheaves by using perverse cohomology.}



\subsection{General results for hyperplane arrangements.}\label{subGA}

Let $$\sA=\lbrace H_{1},\cdots,H_{r} \rbrace$$ be an arrangement of mutually distinct hyperplanes in $\CN$.
Let $$j:U=\CN \setminus \bigcup_{i=1}^r H_i\ra \CN$$ denote the open embedding of the complement. Let $L$ be a rank one $\bC$-local system on $U$. Since $j$ is an affine morphism and quasi-finite,  $Rj_*$ restricts to a functor on the categories of perverse sheaves, hence $Rj_{\ast} (L[n])$  is a perverse sheaf on $\CN$, see \cite[Corollary 5.2.17]{D2}.  

{
The {\it center} of $\sA$ is $\bigcap_{i=1}^r H_i$. The arrangement $\sA$ is called {\it central} if the center is not empty and contains the origin $0$.  The arrangement $\sA$ is called {\it essential} if there is a sub-arrangement $\sB\subset \sA$ such that the center of $\sB$ is a point. For more details see \cite[Definition 2.6]{D3} .
}

An {\it edge} of $\sA$ is either $\bC^n$ or a non-empty set that is an intersection  of hyperplanes in $\sA$. Let $E(\sA)$ denote the set of all edges of $\sA$. 

For each edge $W$, $\sA$ induces a hyperplane arrangement $\sA^W$ in $W$, obtained by intersecting $W$ with the hyperplanes not containing it. To the local system $L$ one can associate a rank one local system $L^W$ on the corresponding complement $U^W=W\setminus\sA^W$.  
On the other hand, $\sA$ also induces a hyperplane arrangement $\sA_W$ in $\CN/W$ and a  rank one local system $L_W$ on the complement $U_W$.  See Section 3 for the precise definitions of the triples $(\sA^W, U^W , L^W)$ and $(\sA_W, U_W, L_W)$. 

We have the following relation between the length of $Rj_*L[n]$ and top-degree cohomology of associated local systems:

\bt \label{arr} Let $L$ be a rank one local system on the complement $U$ of an essential hyperplane arrangement $\sA$ in $\CN$. Then 
\be \label{df}
 [Rj_* L[n]] \geq 
  \sum _{ W \in E(\sA)}  \dim H^{w}(U_W, L_W)\cdot [\IC(W, L^W)],
\ee and hence
\be \label{length}
 \ell(Rj_* L[n]) \geq  \sum _{ W \in E(\sA)}  \dim H^{w}(U_W, L_W), 
\ee where the sums are over all the edges $W$ of $\sA$, $w=\codim W$, and $\IC$ denotes the intersection chain complex. Moreover, (\ref{df}) holds as equality if and only if so does (\ref{length}).
\et


\br \label{positive results}

(a) If $L$ is the constant sheaf, then $L^W$ and $L_W$ are both constant sheaves, and 
\be \label{constant}  [Rj_* \bC_U[n]] = 
  \sum _{ W \in E(\sA)}  \dim H^{w}(U_W)\cdot [\bC_W[n-w]].
\ee
Since the Betti numbers of a complement of a hyperplane arrangement are combinatorial by \cite{OT}, this gives a combinatorial formula for the length of $Rj_{\ast}\bC_U[n]$. This result has been obtained by several people: 
 \cite{Lo, BS, Oa, Pe, BG}.

(b) If $\sA$ is a line arrangement $(n=2)$ or a generic hyperplane arrangement, then (\ref{length}) holds as equality for any rank one local system, see \cite{AB1, AB2}. {Recall that an arrangement in $\bC^n$ is {\it generic} if for every non-empty subset $S\subset\sA$, the intersection of all hyperplanes in $S$ is empty if $|S|>n$ and has codimension $|S|$ if $|S|\le n$.}

 (c) For any rank one local system $L$ on $U$, one has a combinatorial upper bound:
$$ \ell(Rj_* L[n] )\leq (Rj_{\ast}\bC_U[n]).$$
Moreover, this is an equality if and only $L$ is the constant sheaf, see Proposition \ref{upper bound}. This was kindly pointed out to us by a referee.

(d) The inequality (\ref{length}) can be strict, see Example \ref{counter}.

\er

It is easy to see that the length problem for $Rj_* L[n]$ can always be reduced to the central essential hyperplane arrangement case.  A complete combinatorial answer for when length one occurs is thus given by the next result.

Consider the algebraic group $$M_B(U):= \mathrm{ Hom} (H_{1}(U, \Z),\C^{\ast})\cong(\C^{\ast})^{r},$$ the moduli space of rank one local systems on $U$.  A tuple $\underline{t}=(t_{1},\cdots, t_{r}) \in (\C^{\ast})^{r}$ corresponds to the rank one local system $L_{\underline{t}}$ on $U$ with monodromy $t_{i}$ around the hyperplane $H_i$. With this notation, one has:

\bt \label{BLSW}\cite[Theorems 1.2 and 1.5]{BLSW}  If $\sA$ is a central hyperplane arrangement in $\bC^n$ and $L$ a rank one local system on the complement, the following are equivalent:
\begin{itemize}
\item[(a)] $\ell(Rj_*L[n])=1,$
\item[(b)] $Rj_!L[n]=Rj_*L[n]=\IC(\CN,L),$
\item[(c)] 
$
\prod_W(\prod_{W\subset H_i}t_i-1)\neq0,
$
where ${\underline{t}}\in (\bC^*)^r$ with $L=L_{\underline{t}}$, and $W$ are the dense edges. For the definition of dense edges, see the end of Definition \ref{main def}.
\end{itemize} 
\et

We give now a combinatorial answer for when length two occurs in Theorem \ref{arr}:

\bt\label{thrmL2}  If $\sA$ is an essential hyperplane arrangement in $\bC^n$ and $L$ a rank one local system on the complement, the following are equivalent:

\begin{itemize}
\item[(a)] $\ell(Rj_{\ast} L[n])=2$, 
\item[(b)] there exists a dense edge $W$ of $\sA$ such that
$$
[Rj_*L[n]] = [\IC(\CN,L)]+[\IC({W},L^W)],
$$
\item[(c)] there exists a dense edge $W$ of $\sA$ with $\vert \chi (\mathbb{P}(U_{W}))\vert =1$, $\prod_{W\subset H_i}t_i=1$, and $\prod_{W'}(\prod_{W'\subset H_i}t_i-1)\neq 0,$ where ${\underline{t}}\in (\bC^*)^r$ with $L=L_{\underline{t}}$, and $W'$ are the dense edges different from $W$.
\end{itemize}
\et

Here and throughout, $\chi(\_)$ denotes the topological Euler characteristic.

We conjecture that the set of $L$ with $\ell(Rj_{\ast} L[n])=k$ admits a combinatorial description for every $k$, that is, the length function is combinatorial.

\subsection{Plane arrangements.} We  show that (\ref{length}) holds as equality for most of cases in dimension 3.

\bt \label{dim 3} Let $L_{\underline{t}}$ be a rank one local system on the complement of an essential plane arrangement $\sA$ in $\C^3$ corresponding to ${\underline{t}} \in (\C^{\ast})^{r}$. If $t_i\neq 1$ for every $1\leq i \leq r$, then (\ref{length}) holds as equality.
\et

If $t_i=1$ for some $i$, we  have a partial result:

\bp \label{upper} Let $L_{\underline{t}}$ be a rank one local system on the complement of a central essential plane arrangement $\sA$ in $\C^3$ corresponding to $\underline{t} \in (\C^{\ast})^{r}$. Then  $\ell (Rj_* L_{\underline{t}}[3])$ can be computed in terms of the cohomology jump loci $\sV_j^2(U(\sB))$, defined as in (\ref{jump loci}), of the complements of all possible  
sub-arrangements $\sB$ of $\sA$ and all $j\geq 1$. Moreover, 
 \be 
 0\leq \ell(Rj_* L_{\underline{t}}[3]) -\sum _{ W \in E(\sA)}  \dim H^{w}(U_W, L_W) \leq \#\{i \mid t_i =1\}
 \ee
  where the sum is over all the edges $W$ of $\sA$ and $w=\codim W$.
\ep 

This allows us to show that (\ref{length}) can fail to be an equality, see Example \ref{counter}. 

A folklore conjecture is that cohomology jump loci of rank one local systems on complements of hyperplane arrangements admit combinatorial formulas. This would imply that $\ell(Rj_*L[3])$ is also combinatorial, as conjectured above.

\subsection{By-products: intersection cohomology and characteristic cycles.} As an application, when (\ref{df})  holds as equality one can turn it around by deletion-restriction method and induction to provide  formulas for intersection cohomology and characteristic cycles of intersection complexes. This is also related to formulas which have appeared already in Nang-Takeuchi \cite[Theorem 1.1]{NT} for the case of hypersurfaces with isolated singularities, see Remark \ref{rmkNT}.

We restrict for simplicity to central essential plane arrangements in $\bC^3$ and to local systems with the same monodromy around each hyperplane. To fix the notation, for $s\in\bC^*$ let $L_s$ denote the rank one local system on $U$ corresponding to $(s, \ldots,s) \in (\bC^*)^r$. Since $IC(\bC^3,L_1)$ is the (shifted) constant sheaf, we  focus on the case $s\neq 1$.

For a central essential plane arrangement $\sA$ in $\C^{3}$, the set of edges is $$E(\sA)=\lbrace \bC^3, H_{1}, \cdots, H_{r}, \Lambda_1, \cdots, \Lambda_l, 0 \rbrace,$$  
where $\Lambda_j$ are the mutually distinct 1-dimensional edges. Denote by $m_j$ the number of planes of $\sA$ containing $\Lambda_j$.

Let $f=\prod_{i=1}^rf_i$ be an equation defining the hyperplane arrangement $\sA$, with $f_i$ linear polynomials defining the hyperplanes $H_i$.
The smooth variety $$F:= \{x\in \C^3 \mid f(x)=1\}$$
is called the {\it Milnor fiber of $\sA$}. It is classical, see \cite[\S 5.1]{D3} for example, that $F$ is diffeomorphic to the Milnor fiber of $f$ at the origin and the monodromy action on $$H^*(F)=H^*(F,\bC)$$ is semi-simple and has order $r$. 
Let $H^*(F)_s$ denote the monodromy eigenspace of $H^*(F)$ with eigenvalue $s\in\C^*$. Let $\Delta^*(t)$ denote the  characteristic polynomial the monodromy on $H^*(F)$. Then it is known (see \cite[(5.5)]{D3}) that $\Delta^0(t)=t-1$ and $$\dfrac{\Delta^0(t) \Delta^2(t)}{\Delta^1(t)}= (t^r-1)^{\chi(F)/r}.$$
Note that $\chi(F)/r$ is combinatorially determined: $$ \chi(F)/r= -2r+3+\sum_{j=1}^l (m_j-1).$$ So knowing one of $\Delta^1(t)$ and $\Delta^2(t)$ is equivalent to knowing the other one.  

Set  $$\delta_{j}(s)=  \left\{ \begin{array}{ll}
m_j-1  &  \text{ if } s^{m_j} \neq 1, \\
1 &  \text{ if } s^{m_j}=1,\\
\end{array}\right. $$ 
and 
 $$\delta'_j(s)=  \left\{ \begin{array}{ll}
m_j-2 & \text{ if } s^r=1 \text{ and } s^{m_j}=1, \\
0 & \text{ otherwise}.\\
\end{array}\right. $$ 
   
\begin{theorem} \label{cc} Let $\sA$ be a central essential plane arrangement in $\C^3$. Let $1\neq s\in\bC^*$ and let
 $L_s$ be the  rank one local system on the complement of $\sA$ with monodromy $s$ around each plane in $\sA$. Then:

(a) Denoting by $T_{W}^{\ast}\bC^3$ the conormal bundle of the edge $W$,
$$
\begin{array}{rl}
\CC(IC(\bC^3,L_s))  & =
T_{\C^{3}}^{\ast}\C^{3}+\sum_{i=1}^r T_{H_{i}}^{\ast}\C^{3} +\sum_{j=1}^l \delta_j(s) \cdot T_{\Lambda_{j}}^{\ast}\C^{3}+\\
& + \left(\sum_{j=1}^l (\delta_j(s)+\delta'_j(s)) -r+1 -\dim H^2(F)_s\right)  \cdot T_{0}^{\ast}\C^{3}.
\end{array}
$$

(b) In particular, $$\chi(IC(\bC^3,L_s))=  \sum_{j=1}^l \delta'_j(s) { - }\dim H^2(F)_s.$$

(c) Moreover, 
$$
\dim \IH^i(\bC^3,L_s)=\left\{
\begin{array}{cl}
0 & \text{ if }i\neq 1, 2,\\
\dim H^1(U,L_s)=\dim H^1(F)_s & \text{ if } i=1,\\
\dim H^2(F)_s+\dim H^1(F)_s -\sum_{j=1}^l \delta'_j(s)  & \text{ if }i=2.
\end{array}\right.
$$
\end{theorem}

\br The above formulas show that the cycle $\CC(IC(\bC^3,L_s))$ and the numbers $\chi(IC(\bC^3,L_s))$, $\dim H^2(F)_s$,  $\dim H^1(F)_s$, $\dim \IH^1(\bC^3, L_s)$, and $\dim \IH^2(\bC^3, L_s)$ are combinationally determined if and only if one of these six terms is.
\er

 It is an old, still open conjecture that $\dim H^i(F)_s$ are combinatorial.
A recent result of Papadima-Suciu \cite[Theorem 1.2]{PS} proves this conjecture for the case of a central essential plane arrangement $\sA$ in $\bC^3$ that is a cone over a projective line arrangement with multiplicities at most three, that is, $m_j$ is either 2 or 3 for all $j$. More precisely, they have introduced a combinatorial invariant $\beta_3 (\sA) \in \{ 0,1,2\}$ of $\sA$ such that $$\Delta^1(t)=(t-1)^{r-1}(t^2+t+1)^{\beta_3(\sA)}.$$ One obtains then the following combinatorial formulas:

\begin{theorem} \label{PS} Let $\sA$ be a central essential plane arrangement in $\C^3$ that is a cone over a projective line arrangement with multiplicities at most three. Let $1\neq s\in\bC^*$ and let
 $L_s$ be the  rank one local system on the complement of $\sA$ with monodromy $s$ around each plane in $\sA$. Let $n_3(\sA)$ denote the number of 1-dimensional edges with multiplicity 3. Then:

(a)   $$
\CC(IC(\bC^3,L_s)) = 
T_{\C^{3}}^{\ast}\C^{3}+\sum_{i=1}^r T_{H_{i}}^{\ast}\C^{3} + \bigstar,
$$
where
\begin{align*}
\bigstar= 
\left\{ \begin{array}{ll}
{\sum_{m_j=2}  T_{\Lambda_{j}}^{\ast}\C^{3} +2\sum_{m_j=3}  T_{\Lambda_{j}}^{\ast}\C^{3}+ (\binom{r-1}{2}-n_3(\sA)) \cdot T_{0}^{\ast}\C^{3}} &  \text{ if }s^r\neq 1 \text{ and } s^3\neq 1, \\
\sum_{m_j=2} T_{\Lambda_{j}}^{\ast}\C^{3} +2\sum_{m_j=3}  T_{\Lambda_{j}}^{\ast}\C^{3}+ (r-2)\cdot T_{0}^{\ast}\C^{3},  &  \text{ if }s^r=1\text{ and } s^3\neq 1, \\
\sum_{j=1}^l  T_{\Lambda_{j}}^{\ast}\C^{3}+(\binom{r-1}{2} -2n_3(\sA)) \cdot T_{0}^{\ast}\C^{3},  & \text{ if }s^r\neq 1\text{ and } s^3=1, \\
\sum_{j=1}^l  T_{\Lambda_{j}}^{\ast}\C^{3}+(r-2-\beta_3(\sA))\cdot T_{0}^{\ast}\C^{3} , &  \text{ if }s^{r}=1\text{ and } s^3=1.\\
\end{array}\right.
\end{align*}

(b) $$
\chi(IC(\bC^3,L_s))=  \left\{ \begin{array}{ll}
n_3(\sA)-\binom{r-1}{2}  &  \text{ if } s^r=1\text{ and }s^3\neq 1, \\
2n_3(\sA)-\beta_3(\sA)-\binom{r-1}{2}  & \text{ if } s^{r}=1\text{ and } s^3=1,\\
0 & \text{ otherwise}. \\ 
\end{array}\right.
$$ 

(c)
$$\IH^0(\bC^3, L_s)=0=\IH^3(\bC^3, L_s),$$
$$\dim \IH^1(\bC^3, L_s)= \left\{ \begin{array}{ll}
\beta_3(\sA) &  \text{ if } s^r=1\text{ and } s^3=1,\\
0 & \text{ otherwise}, \\ 
\end{array}\right. $$
$$ \dim \IH^2(\bC^3, L_s)= \left\{ \begin{array}{ll}
\binom{r-1}{2} -n_3(\sA)  &  \text{ if } s^r=1\text{ and } s^3\neq 1, \\
2\beta_3(\sA)+\binom{r-1}{2} -2n_3(\sA) &  \text{ if } s^{r}=1\text{ and } s^3=1,\\
0 & \text{ otherwise}. \\ 
\end{array}\right. $$

{{(d)  \begin{align*}
\ell( Rj_*L_s[3])= 
\left\{ \begin{array}{ll}
1 &  \text{ if }s^r\neq 1 \text{ and } s^3\neq 1, \\
1+\binom{r-2}{2} -n_3(\sA)  &  \text{ if }s^r=1\text{ and } s^3\neq 1, \\
1+ n_3(\sA)  & \text{ if }s^r\neq 1\text{ and } s^3=1, \\
1+ \binom{r-2}{2} +\beta_3(\sA) &  \text{ if }s^{r}=1\text{ and } s^3=1.\\
\end{array}\right.
\end{align*}
}
}\end{theorem}


\subsection{$\sD$-modules.} Due to the Riemann-Hilbert correspondence, every statement made so far can be restated (and proved) in terms of $\sD$-modules only. Relevant articles on the questions addressed in this article are sometimes written in terms of $\sD$-modules, e.g. \cite{AB1,AB2, BG, Gin, Oa}. To make this article accessible to $\sD$-module theorists, and conversely, to justify the topological statements one derives from the cited works, we recall now the algebraic counterparts of the two main topological objects studied in this article.

Let $\sD$ be the sheaf of linear algebraic differential operators on $\bC^n$. Let $f=\prod_{i=1}f_i$ be an equation defining the hyperplane arrangement $\sA$, with $f_i$ linear polynomials defining the hyperplanes $H_i$. Let $L_{\underline{t}}$ with ${\underline{t}}\in(\bC^*)^r$ be a rank one local system on $U=\{f\ne 0\}$. Fix ${\underline{\al}}\in\bC^r$ such that ${\underline{t}}=\exp(2\pi i {\underline{\al}})$. Then the monomorphism of perverse sheaves $$IC(\CN,L_{\underline{t}})\subset Rj_*L_{\underline{t}}[n]$$ corresponds  to the monomorphism of $\sD$-modules generated by applying the operators in $\sD$ in an obvious way to symbols as follows:
$$\sD\cdot \prod_{i=1}^rf_i^{\al_i+k}\subset\sD\cdot \prod_{i=1}^rf_i^{\al_i-k}=\cO[\prod_if_i^{-1}]\prod_if_i^{\al_i}$$ with $k\in\bN$, $k\gg 0$, where
$\cO$ is the sheaf of regular functions on $\CN$, see \cite{B-ls}.

\subsection{Method and organization.}
In this paper, varieties are complex algebraic varieties, and all cohomology groups are taken with $\bC$ coefficients unless otherwise stated.

For the proofs we occasionally use a few non-elementary facts: the structure and propagation of cohomology jump loci of rank one local systems on complements of hyperplane arrangements \cite{DSY}, the fact that length jump loci are absolute $\bQ$-constructible sets \cite{BGLW}, a bit of mixed Hodge theory \cite[Lemma 2.18]{BFNP} saying that the intermediate extension functor is exact upon certain weight conditions, and as already mentioned, the main result of \cite{PS}.

In Section \ref{secLow} we give a general lower bound for the lengths of certain perverse sheaves. In Section \ref{secGHA} we refine this calculation for the case of hyperplane arrangements and prove Theorems \ref{arr} and \ref{thrmL2}. In Section \ref{secPl} we specialize to dimension three and prove Theorem \ref{dim 3} and Proposition \ref{upper}. In Section \ref{secIC} we prove Theorems \ref{cc} and \ref{PS}.

\subsection{Acknowledgement.} We thank B. Wang for his help with various questions related to this paper. We thank the referees for helping us improve the paper. The first author was partly supported by the grants STRT/13/005 and Methusalem METH/15/026 from KU Leuven, and G0B2115N, G097819N, and G0F4216N from the Research Foundation - Flanders. The second author was supported by the ERCEA 615655 NMST Consolidator Grant and also by the Basque Government through the BERC 2018-2021 program and Gobierno Vasco Grant IT1094-16, by the Spanish Ministry of Science, Innovation and Universities: BCAM Severo Ochoa accreditation SEV-2017-0718.

\section{A lower bound for the length of perverse sheaves}\label{secLow}

The goal of this section is to prove Theorem \ref{lower} which gives a lower bound on the length of certain perverse sheaves related with any hypersurface. In the next sections, we will specialize to hyperplane arrangements.

\subsection{Length of perverse sheaves.}
 Let $X$ be  a complex  algebraic variety.  We denote by $D^b_c(X)$ the derived category of bounded $\bC$-constructible sheaves on $X$, and by $\Perv(X)$ the abelian category of $\bC$-perverse sheaves. 
 $\Perv(X)$ is an artinian and noetherian  category, see \cite[Theorem 4.3.1]{BBD}.   In other words, every perverse sheaf $P$ has a finite length composition series $$ 0=P_{0}\hookrightarrow P_{1} \hookrightarrow \cdots \hookrightarrow P_{m}=P $$
 for which the quotients $P_{i}/P_{i-1}$  are simple, that is, of the form $\IC(\overline{S},L)$.  Here  $S$ is a connected stratum in a Whitney stratification of $X$ with respect to which $P$ is constructible, $L$ is an irreducible local system on $S$, and $$\IC(\overline{S},L):=(i_{\overline{S}})_*(j_{S})_{!*}(L[\dim S])$$ denotes the corresponding intersection chain complex on the closure $\overline{S}$, where $j_{S}:S\ra \overline{S}$ and $i_{\overline{S}}:\overline{S}\ra X$ are the natural inclusions, and $j_{!*}$ is the intermediate extension functor from \cite[D\'efinition 1.4.22]{BBD} which we recall below as well. The simple perverse sheaves $P_i/P_{i-1}$ are called the {\it decomposition factors} of $P$.

Let $G(X)$ be the Grothendieck group of $\Perv(X)$. Then we have that in $G(X)$,
$$ [P]= \sum_{i=1}^m [P_i/P_{i-1}]$$
and $\ell(P)=m$. In particular, $[P]$ and $\ell(P)$ do not depend on of the choice of composition series for $P$.
If $R$ is a sub or quotient perverse sheaf of $P$, then it is clear that $[P]\geq [R]$.


\subsection{Intermediate extension functor.}

Let $j: U \to X$ be a locally closed embedding such that $X= \overline{U}$. Given a perverse sheaf $P$ on $U$, the natural map $j_! P\to Rj_* P$ induces a map on perverse cohomology $^p \cH^{0} (j_! P) \to { }^p \cH^{0} (Rj_* P) $. The {\it intermediate extension} $j_{!*} P$ is defined as the image of this morphism in  $\Perv(X)$. If $U$ is smooth of dimension $n$ and $L$ is a local system on $U$, then $\IC(X,L)=j_{!*}(L[n])$.

The following facts about the intermediate extension functor can be found in \cite{DM}.  The intermediate extension functor behaves well only when it comes to the simple perverse sheaves, that is, it takes a simple perverse sheaf into a simple perverse sheaf. In general, $j_{!*}: \Perv(U) \to \Perv(X)$ is not exact in the following way.  Let $$0 \to P \overset{a}{\to} Q \overset{b}{\to} R\to 0$$ be a short exact sequence in $\Perv(U)$. Then  $j_{!*}$ preserves  injections and projections, but exactness in the middle can fail in general. This implies:

\bp \label{p1.1} Let $j: U \to X$ be a locally closed embedding such that $X=\overline{U}$, where $X$ is irreducible. Consider a short exact sequence of perverse sheaves on $U$: 
\be \label{1.0} 0 \to P \overset{a}{\to} Q \overset{b}{\to} R\to 0
\ee Then we have the following inequality:
\be  \label{1.1}
[j_{!*} Q] \geq [j_{!*} P] + [j_{!*} R], \ee hence
\be  \label{1.2}
\ell(j_{!*} Q) \geq \ell(j_{!*} P) + \ell(j_{!*} R).
\ee
Moreover, $j_{!*}$ is exact to the short exact sequence (\ref{1.0}) if and only if (\ref{1.1}) or (\ref{1.2}) holds as equality.
\ep

\bc \label{c1.1}  Let $j: U \to X$ be a locally closed embedding such that $X=\overline{U}$, where $X$ is irreducible. { Assume that $P\in \Perv(U)$ with  $[P]=\sum_{[Q]} a_{[Q]}(P)\cdot[Q]$ in $G(U)$, with $Q$ simple perverse sheaves. Set $j_{!*} [P]:= \sum_{[Q]} a_{[Q]}(P)\cdot[j_{!*} Q]$ in $G(X)$. Then we have that }
\be  \label{1.3} [j_{!*}P] \geq j_{!*} [P] 
\ee hence \be \label{1.4} \ell(j_{!\ast}P) \geq \ell(P) 
\ee
\ec
\begin{proof}
Assume that $P$ has length $m$. Then the composition series of $P$ gives us $(m-1)$ short exact sequences of perverse sheaves. The claim follows from  (\ref{1.1}), (\ref{1.2}), and the fact that $j_{!*}$ takes simple perverse sheaves into simple perverse sheaves. 
\end{proof}
\bex \label{ex1.1} \cite[Example 2.7.1]{DM} Let $L$ be a rank $2$ local system on the punctured complex line  $\C^{\ast}$ defined by the automorphism $\bigl(\begin{smallmatrix}
1 & 0\\
1 & 1
\end{smallmatrix} \bigr) \in \mathrm{GL}_{2}(\C)$.  
One has a short exact sequence of perverse sheaves:
$$0\to \bC_{\bC^*}[1] \to L[1] \to \bC_{\bC^*}[1]\to 0  ,$$
where $\C_{\C^{\ast}}$ is the rank $1$ constant sheaf on  $\C^{\ast}$. Let $j$ denote the open inclusion from $\C^{\ast} $ to $ \C$. As shown in \cite[Example 2.7.1]{DM}, $j_{!*}$ is not an exact functor for this short exact sequence. 

Note that $[L[1]]=2 \cdot [\C_{\bC^*}[1]]$ and $\ell(L[1])=2$.   It is easy to compute that $[j_{!*}(L[1])]=2\cdot [\C_{\C}[1]]+1 \cdot[\C_0]$, hence $\ell(j_{!*}(L[1]))=3$. Thus strict inequalities can happen in Corollary \ref{c1.1}.
\eex

\subsection{A lower bound.} Let $X$ be an irreducible algebraic variety of dimension $n$, and $Z$ a closed proper subvariety of  $X$, such that the complement $U=X \setminus Z$ is smooth.  Let $j:U\ra X$ be the open embedding. Let $L$ be a local system on $U$. We assume that $Rj_{\ast} (L[n])$  is a perverse sheaf on $X$ (e.g., $j$ is an affine morphism). Then $j_{!*} (L[n])$ is a perverse subsheaf of $Rj_{\ast} (L[n])$. If $L$ is irreducible, then $j_{!*} (L[n])$ is simple, hence one can use the length of  $Rj_{\ast} (L[n])$ to measure the difference  between these two perverse sheaves. The following is probably known, but we could not find a reference:

\bt \label{lower} With the above notations, assume that $Rj_{\ast} (L[n])$  is a perverse sheaf on $X$. Fix a Whitney  stratification $\sG$ of $X$ such that $Rj_{\ast} (L[n])$ is constructible with respect to this stratification. Then: \be \label{1.5}
 [Rj_* L[n]] \geq  \sum _{S\in \sG}  [\IC(\overline{S}, \K_{S})] 
\ee hence
\be \label{1.6}
 \ell(Rj_* L[n]) \geq  \sum _{S\in \sG}  \ell  (\IC(\overline{S}, \K_{S})), 
\ee where $\K_{S}=(R^{s}j_{\ast}L)\vert _{S}$ is a local system on $S$, and  $s$ is the codimension of $S$.
\et

\begin{proof}
Let $U_{i}$ denote the union of all the strata in $\sG$ of codimension at most $i$. Then we have the following sequence of maps: 
$$ U=U_{0} \overset{j_{1} }{\hookrightarrow} U_{1} \overset{j_{2} }{\hookrightarrow} \cdots    \overset{j_{n} }{\hookrightarrow} U_{n}=X.$$
Set $$j_{i+m,\cdots,i}= j_{i+m}\circ j_{i+m-1}\cdots \circ j_{i}:U_{i-1}\hookrightarrow U_{i+m}.$$ In particular, $j_{n,\cdots,1} =j$. The set $U_{i}$ is open in $X$, so  $(j_{n,\cdots,i+1})^{-1}$ is a perverse $t$-exact functor, see \cite[Theorem 5.2.4(iv)]{D2}. Hence 
\be\label{eqMM}
R(j_{i,\cdots,1})_{\ast} (L[n])= (j_{n,\cdots,i+1})^{-1} Rj_{*}L[n]
\ee is also perverse for any $1\leq i\leq n$.  Thus
$(j_{i})_{!\ast}  R(j_{i-1,\cdots,1})_{\ast}( L[n])$ is a perverse subsheaf of $R(j_{i,\cdots,1})_{\ast} L[n]$. We have  $n$ short exact sequences of perverse sheaves: 
\begin{gather}
0 \to (j_{1})_{!\ast} (L[n]) \to   R(j_{1})_{\ast} L[n] \to Q_{1} \to 0   \notag \\
\vdots  \nonumber \\
 0 \to (j_{i})_{!\ast}  R(j_{i-1,\cdots,1})_{\ast} (L[n]) \to   R(j_{i,\cdots,1})_{\ast} L[n] \to Q_{i} \to 0  \label{si} \\
\vdots \nonumber \\
0 \to (j_{n})_{!\ast}  R(j_{n-1,\cdots,1})_{\ast} (L[n]) \to   R(j_{n,\cdots,1})_{\ast} L[n] \to Q_{n} \to 0   \notag 
\end{gather} where $Q_i$ denotes the corresponding quotient perverse sheaf on $U_{i}$. It is clear that  $Q_{i}\vert_{U_{i-1}}=0.$

 Recall Deligne's construction (\cite[Proposition 2.1.11]{BBD}), for any $1\leq i \leq n$: 
$$ (j_{i})_{!\ast}  R(j_{i-1,\cdots,1})_{\ast}( L[n]) = \tau_{\leq  i-1-n} R(j_{i,\cdots,1})_{\ast} (L[n]) ,$$ where the truncation functor $\tau$ is with respect to the standard $t$-structure.
 Since the complex $R(j_{i,\cdots,1})_{\ast}( L[n])$ is a perverse sheaf on $U_{i}$, one has $\tau_{> i-n} R(j_{i,\cdots,1})_{\ast} (L[n]) =0$. Therefore we have the following quasi-isomorphism in $D^{b}_{c}(U_{i})$:
\begin{align*}
 Q_{i} & =\tau_{\geq  i-n} R(j_{i,\cdots,1})_{\ast} (L[n])  =  R^{i-n}(j_{i,\cdots,1})_{\ast} (L[n]) [n-i] \\
        & = (R^{i}(j_{i,\cdots,1})_{\ast} L)[n-i]  = (( R^{i}j_{\ast} L)\vert_{ U_{i}})[n-i].
\end{align*}  Thus there is a quasi-isomorphism  
$$Q_{i} =\bigoplus_{S} \K_{S}[n-i],$$  where the direct sum is over all the strata $S$ in $\sG$ of codimension $i$.
In particular,   \[ \begin{aligned}
[ Rj_{\ast} L[n] ]
& = [(j_{n})_{!\ast}  R(j_{n-1,\cdots,1})_{\ast} (L[n])]+ [Q_{n}] \\
& \geq [(j_{n,n-1})_{!\ast} ( R(j_{n-2,\cdots,1})_{\ast} (L[n])]+ [(j_{n})_{!*} Q_{n-1}] +[Q_{n}]  \\
& \geq   \ldots \geq [(j_{!\ast}(L[n])]+\sum _{i=1}^{n}  [(j_{n,\cdots,i+1})_{!*} Q_{i}]   \\
& =\sum _{S\in \sG}  [\IC(\overline{S}, \K_{S})] 
\end{aligned} \]
where the inequalities follow from (\ref{1.1}) and applied to (\ref{si}) repeatedly.  
 \end{proof}

\br\label{rmkNT} (a) When $n=1$,  (\ref{1.5}) and (\ref{1.6}) are equalities as it can be seen from the proof, since (\ref{si}) consists of only one exact sequence in this case. For example, in the case of Example \ref{ex1.1} one has $\ell(Rj_*L[1])=4$.

For $n>1$, the following are equivalent: (\ref{1.5}) is an equality; (\ref{1.6}) is an equality; $(j_{n,\cdots,i+1})_{!*}$ is exact on the short exact sequences (\ref{si}) for any $1\leq i\leq n-1$.

(b) $Rj_* (L[n])=j_{!*} (L[n])$ if and only if $\K_{S}=0$ for any stratum $S\subset Z$. 
  
(c)  When $X$ is smooth and $Z$ is a hypersurface, detecting the simpleness of $Rj_*(L[n])$ for rank one local systems on $U$ was studied in \cite{BLSW}.

(d) If $X=(\CN,0)$ is the germ of the origin in $\CN$, $n\geq 3$,  $Z$ is a hypersurface with only isolated singularities at the origin, and $L$ is a rank one local system on $U=X-Z$, the above proof also works for this analytic case. In particular, we claim that (\ref{1.5}) is an equality in this case. Here we take the stratification of $X$ with three strata $U$, $Z-\{ 0\}$ and $\{0\}$. Note that  $\pi_1(U)\cong \Z$ due to the Milnor fibration and the well-known fact the Milnor fiber is $(n-2)$-connected for isolated singularity case. If $L$ is not the constant sheaf, the above proof proves the claim. On the other hand, if $ L$ is the constant sheaf, a direct computation can be used to show $(j_n)_{!*}$ is exact on the short exact sequence (\ref{si}) in this case, hence the claim follows.   Moreover, once (\ref{1.5}) holds as equality, one can give a formula for the characteristic cycles of $\IC(X,L)$ using the approach we give later in Section \ref{secIC}. This formula is not new, it coincides with the one given by  \cite[Theorem 1.1]{NT}. We leave the details out since they are not the focus of this paper.
\er

\section{General  results for hyperplane arrangements}\label{secGHA}

We specialize now to hyperplane arrangements and prove Theorems  \ref{arr} and \ref{thrmL2}.

\subsection{Terminology.}\label{subTerm}
Let $\sA=\lbrace H_{1},\cdots,H_{r} \rbrace$ be an arrangement of mutually distinct hyperplanes in $\CN$, with complement $$j:U=\CN \setminus \sA\hookrightarrow \bC^n.$$ 
In addition to the terminology specific to hyperplane arrangements already introduced in \ref{subGA}, we will also need the following standard definitions. We say that $\sA$ is {\it decomposable} if there exist nonempty disjoint subarrangements $\sA_1$ and $\sA_2$ with
$\sA = \sA_1\cup\sA_2$ and, after a linear coordinate change, the defining equations for
$\sA_1$ and $\sA_2$ have no common variables. If $\sA$ is central, this is equivalent to $\chi (\mathbb{P}(U)) \neq 0$, where $\pp(U)$ is the complement of the projectivization of the arrangement; hence this is a combinatorial notion, see \cite{STV}.

Let $M_B(U)$ be the moduli space of rank one local systems on $U$.  Then  $$M_B(U)= \mathrm{Hom}(H_{1}(U, \Z),\C^{\ast})=(\C^{\ast})^{r}$$ as algebraic group. The last canonical isomorphism identifies a local system $L$ with ${\underline{t}}=(t_{1},\cdots,t_{r}) \in (\C^{\ast})^{r}$, where $t_i$ is its monodromy around the hyperplane $H_i$. In the rest of this paper, $L=L_{\underline{t}}$ will always denote a rank one local system on $U$ and ${\underline{t}}$ is omitted if the context is clear.

\bd  \label{main def} {Let $\sA$ be an essential hyperplane arrangement.} To a rank one local system $L=L_{\underline{t}}$ on $U$ and a fixed edge $W\in E(\sA)$, we associated two tuples $(\sA_W,U_W,L_W)$ and $(\sA^W,U^W, L^W)$ as follows. Firstly, after renumbering the hyperplanes, let $H_1,\ldots,H_k$ denote all the hyperplanes in $\sA$ containing $W$.

\begin{itemize}

\item Define $\bC^n/W$ the quotient vector space obtained by moving the origin in  $W$ via a change of coordinates.   Then $\CN =\CN/W \times W$ as affine spaces.

\item Define the arrangement $ \sA_{W} =\lbrace H_{1}/W, \ldots, H_{k}/W \rbrace$ in $\CN/W$. In particular, $\sA_W$ is always a central essential arrangement. Denote its  complement by $U_{W}= (\CN/W)\setminus \sA_{W}$. 

\item  Define the arrangement $ \sA^{W}=\lbrace W\cap H_{k+1},\ldots ,W\cap H_r\rbrace$ in $W$. { It may happen that $ W\cap H_{i}=\emptyset$ or $W\cap H_{i}=W\cap H_{j}$ for $i,j>k$ and $i\neq j$. The essential assumption for $\sA$ implies that $\sA^W$ is non-trivial, if $\dim W \geq 1$.}  Its complement is $U^{W}= W- \sA^{W}$. 
If $W$ is a point, then $\sA^{W}=\emptyset$.

\item For any $x\in U^{W}$, let $U_x$ be the complement of $\sA$ in a small ball in $\bC^n$ centered at $x$. Then  $U_{x}$ is homotopy equivalent to $U_{W}$. Restricting $L$ to $U_x$, one has  a corresponding rank one local system on $U_W$, which we denote by $L_W$. In terms of monodromy $t\in(\bC^*)^r$,  $L_W$ is the projection of $t$ onto the first $k$ coordinates in $M_B(U_W)=(\bC^*)^k$.

\item
We define a rank one local system $L^{W}$ on $U^{W}$ by setting the monodromy around an edge $D$, which has to be an hyperplane in the arrangement $\sA^W$, to be the non-zero complex number $\prod_{j=1}^{l} t_{i_{j}}$, where $H_{i_j}$ are all the hyperplanes  in $\sA$ containing $D$ but not $W$, that is,
$D=W\cap H_{i_{1}}\cap \cdots \cap H_{i_{l}}$ with $i_j>k$. 
\end{itemize}
If $W=\CN$, then $(\sA_W,U_W,L_W)=(\emptyset, \{0\}, \bC_{\{0\}})$ and $(\sA^W,U^W, L^W)=(\sA,U,L)$.

An edge $W\in E(\sA)$ is {\it dense} if the central subarrangement $\sA_{W}$ is {indecomposable} and, by convention, $W\neq\bC^n$. For the definition of indecomposable arrangement, see \cite[Definition 2.5]{D3}.
 In particular, density is a combinatorial notion.
\ed

\subsection{The lower bound.} For the proof of Theorem \ref{arr}, we first show:

\bp \label{pushfoward} Let $L$ be a rank one local system on the complement $U$ of an essential hyperplane arrangement $\sA$ in $\CN$. For any edge $W\in E(\sA)$ and $i$ an integer, 
\be \label{2.1}  (R^i j_{\ast} L)\vert_{U^{W}} =\oplus L^{W},
\ee  where the number of copies of $L^W$ is $\dim H^i(U_W,L_W)=\dim H^i(U_{x}, L)$ for any $x\in U^{W}$.
\ep 

\begin{proof}
If $W$ is a point, then $U^{W}=W$ is also a point, hence  $L^W$ is the skyscraper sheaf  and the claim is clear. If $W=\CN$ the claim is also clear. 

Now we assume that $1\leq \dim W <n$. 
We can assume $0\in W$, $W=\bigcap_{i=1}^{k} H_{i}$, and $W\not\subset H_i$ if $i>k$.

Consider the following trivial fibrations:
$$\xymatrix{
 B_{\varepsilon}\cap U_{W}   \ar[rd]   &  U_{W} \ar[rd]     &  U_{W} \ar[rd] & \\
                             &   Y  \ar[d] \ar[r]              & U_{W}\times U^{W} \ar[r]  \ar[d] & \CN\setminus  \bigcup_{i=1}^{k} H_{i} \ar[d]\\
                             &   T( U^{W}) \ar[r] &  U^{W} \ar[r] &ã W
}$$
where: the third fibration is the restriction of the projection map $\CN=\CN/W\times W\ra W$, the second one is the pull back of the third one, and the first one we explain now. We shrink both the fibre and the base space in the second fibration  such that  the resulting total space $Y$ is homotopy equivalent to $U^{W} \times U_{W}$ and $Y \subset U$. Set $\codim W=w$. View $U^{W}$ as the complement of a hyperplane arrangement in $\mathbb{P}^{n-w}.$ As shown by Durfee  \cite{Dur}, there exists an isotopic neighbourhood of $U^{W}$, denoted by $T(U^{W})$, such that $T(U^{W}) \subset U^{W}$ is a homotopy equivalence. See \cite[page 149]{D1} for the concrete construction. Moreover, for this construction, $$ \lbrace \Vert x-y \Vert \mid x\in T(U^{W}), y\in H_{i}, k+1\leq i \leq r   \rbrace$$
has a positive lower bound, say $\alpha$. As long as we take $0<\varepsilon<\alpha$,  $B_{x,\varepsilon} \setminus U^{W} \subset U $ for any $x\in  T(U^{W})$, where $B_{x,\varepsilon}$ is an open ball in $\CN$ centred at $x$ with radius $\epsilon$.   Recall that $\sA_W$ is a central arrangement. 
Let $B_\epsilon$ denote the open ball in $\CN/W$ with radius $\epsilon$ centered at the origin. 
Then $B_\epsilon \cap U_W$ is homotopy equivalent to $U_W$. 
Taking $Y= T(U^W)\times(B_{\varepsilon}\cap U_{W}) $, it is clear that $Y$ is homotopy equivalent to $U^{W} \times U_{W}$  and $Y\subset U$.

Consider the following composition of natural maps:
\begin{center}
$\pi_{1}(U_{W})\times \pi_{1}( U^{W})= \pi_{1}(U_{W}\times U^{W})=\pi_{1}(Y) \to \pi_{1}(U) \to (\C^{\ast})^{r} $
\end{center}
where the last map is the character map  defining the local system $L$. 
Using the pull back functor, we have two rank one local system $L_{W}$ on $U_{W}$ and  $L^{W}$ on $U^{W}$, respectively, as we defined before. 
Then $$L\vert _{Y} = (L_{W}\otimes_{\bC} L^{W}) \vert_{Y}, $$
where $L_{W}\otimes_{\bC} L^{W}$ means $(p_{1}^{-1} L_{W}) \otimes_{\bC} (p_{2}^{-1}L^{W})$ with $p_{i}$ the  projection maps from $U_{W}\times U^{W}$.

 Consider the following commutative diagram of inclusions:
\begin{center}
$\xymatrix{
& Y \ar[r]    \ar[d]^{j^{\prime}}         & U \ar[d]^{j}  \\
T(U^{W}) \ar[r] & T(U^{W})\times B_{\varepsilon}         \ar[r]      & \CN
}$
\end{center}
Here $T(U^W)$ is viewed as $T(U^W)\times 0$, where $0$ is the center of the ball $B_\epsilon$.
Then \begin{center}
$ (Rj_{\ast}L) \vert_{T(U^{W})}=(Rj^{\prime}_{\ast} (L\vert_{Y})) \vert_{T(U^{W})}=(Rj^{\prime}_{\ast} ((L_{W}\otimes_{\bC} L^{W}) \vert_{Y}) ) \vert_{T(U^{W})},$
\end{center}
where the first quasi-isomorphism follows from the fact that $Y$ is open in $U$.
Due to the product structure and the K\"{u}nneth formula, we have that $(R^{i}j_{\ast}L) \vert_{T(U^{W})}= \oplus L^{W}\vert_{T(U^W)}$, where the number of copies is $\dim H^{i}(U_{W}, L_{W})$. Note that $(R^{i}j_{\ast}L) \vert_{U^{W}}$ is a local system. Now $T(U^W)$ is open in $U^W$ and $T(U^W) $ is homotopy equivalent to $U^W$. So $(R^{i}j_{\ast}L) \vert_{U^{W}}=\oplus L^{W}$, where the number of copies is $\dim H^{i}(U_{W}, L_{W})$.
\end{proof}

\begin{proof}[{\bf Proof of Theorem \ref{arr}}]
The claim follows from Proposition \ref{pushfoward} and Theorem \ref{lower} directly.
\end{proof}



\br  Bibby (\cite[Lemma 3.1]{Bib}) proved a similar claim as in Proposition \ref{pushfoward}  in a more general set-up, but only for the constant sheaf case.  
\er

\br\label{remQ1}
If $W$ is one of the hyperplanes in $\sA$, then $U_W$ is homotopy equivalent with a circle, hence the following are equivalent:  (i) the monodromy of $L$ around $W$ is trivial; (ii) $L_W$ is trivial; (iii) $H^1(U_W,L_W)\neq 0$. By Proposition \ref{pushfoward}, this is then further equivalent to: (iv) $\sL_W\neq 0$,  in the notation of Theorem \ref{lower}.
\er

\begin{que} \label{question 3.6}  Based on Theorem \ref{arr}, one may also ask if every decomposition factor of $Rj_*L[n]$ is always an intermediate extension of local systems of rank precisely one.
\end{que}

\subsection{Length two.} For the proof of Theorem \ref{thrmL2}, we will need some facts. { The {\it cohomology jump loci of $U$} are defined by 
\be \label{jump loci} \sV^{i}_j(U) :=  \{L \in M_B(U) \mid   \dim H^i(U,  L) \geq j\}.
\ee
The cohomology jump loci are homotopy invariants of $U$ and Zariski closed sub-varieties in $M_B(U)=(\bC^*)^r$.}
In this article we do not consider the possibly non-reduced scheme structure on these sets.
Set  $\V^{i}(U)=\V_{1}^{i}(U)$ and  
$\V(U)= \cup_{i\geq 0} \V^{i}_1(U).$ 

If $\sA$ is central indecomposable then $\V(U)=\{\prod_{i=1}^rt_i=1\}$, see \cite[Proposition 3.25]{B-ls}.
If $\sA$ is central essential then $\sV(U)=\sV^n(U)$, by  \cite[Theorem 1.3]{DSY}; we call this the {\it propagation property}.  Hence if $W$ is a dense edge,
$$\V(U_{W})=\{\prod_{W\subset H_i}t_i=1\}\subset M_B(U_W)=(\bC^*)^{\#\{i\mid W\subset H_i\}}.$$ Setting
$$
\W(W)=\{L\in M_B(U)\mid L_W\in \V(U_W)\},
$$ one has $\W(\bC^n)=M_B(U)$, and, if $W$ is dense, $\W(W)=\{\prod_{W\subset H_i}t_i=1\}\subset (\bC^*)^r$.
Set $$\W(W)^{\circ}=\W(W)\cap\left ((\bC^*)^r\setminus \bigcup_{W^{\prime} \neq W,\bC^n} \W(W^{\prime})\right)$$
where the  union is over all the edges $W'$ different from $W$ and $\bC^n$. 

\bl\label{lemWo} Let $W$ be an edge. Then
$$\W(W)^{\circ}=\W(W)\cap \left ((\bC^*)^r\setminus \bigcup_{W^{\prime} \neq W} \W(W^{\prime})\right)$$
where the  union is over all the dense edges $W'$ different from $W$. In particular,  if $W$ is dense, then
$$
\W(W)^{\circ}=\left\{\prod_{W\subset H_i}t_i=1\right\}\cap \left\{\prod_{\text{dense } W'\ne W}(\prod_{W'\subset H_i}t_i-1)\neq 0 \right \}\subset(\bC^*)^r,
$$
and
$$
\W(\bC^n)^\circ=\left\{\prod_{\text{dense } W'}(\prod_{W'\subset H_i}t_i-1)\neq 0 \right \}\subset(\bC^*)^r.
$$ 
\el
\begin{proof} If $W'\neq\bC^n$ is not a dense edge, then there exists a finite collection of dense edges $\lbrace W_{i} \rbrace_{1\leq i \leq l}$ such that $U_{W'}=U_{W_{1}} \times \cdots \times U_{W_{l}}$. Then, by K\"unneth formula,  $\W(W')=\bigcap_{i=1}^{l} \W(W_{i}).$ This implies the first claim. The rest follows from the explicit formulas for dense edges given above. 
\end{proof}

\begin{lemma}\label{lemNonRes}
If $W$ is dense and $L\in \W(W)^\circ$, then $(\sA^W, U^W,L^W)$  satisfies the equivalent conditions of Theorem \ref{BLSW}.
\end{lemma}
\begin{proof}
 Fix a dense edge $D\in E(\sA^{W})$. Let $\{ H_{i_1}, \cdots, H_{i_l} \}$ be all the hyperplanes in $\sA$ containing $D$ but not $W$, hence $D=W\cap H_{i_{1}}\cap \cdots \cap H_{i_{l}}$ with $W\nsubseteq H_{i_j} $ for any $1\leq j\leq l$. We only need to show that $L\in \W(W)^\circ $ implies that $\prod_{j=1}^l t_{i_j} \neq 1$. 
 
  There are two possible cases. In the first case, $D$ is also a dense edge in $E(\sA)$. Then $L\in \W(W)^\circ $ implies, by Lemma \ref{lemWo},  that $\prod_{W\subset H_i}t_i=1 $ and $\big(\prod_{W\subset H_i}t_i \big)\big(\prod_{j=1}^l t_{i_j}\big) \neq 1,$ hence the claim follows.
  
In the second case, $D$ is not a dense edge in $E(\sA)$. Then there exists a finite collection of dense edges $\lbrace W_{j} \rbrace_{1\leq j \leq s}$ of $\sA$ with $s\geq 2$ such that $U_{D}=U_{W_{1}} \times \cdots \times U_{W_{s}}$. In particular, $D=\bigcap_{j=1}^s W_j$.
  Since $W$ is a dense edge in $\sA$, it follows that $W$ is also a dense edge for exactly one of these arrangements $\sA^{W_j}$, say $\sA^{W_1}$. In particular, $W_1\subset W$. Since $D$ is a dense edge in $E(\sA^W)$,   it follows that $W=W_1$ and $s=2$. 
  Then $W_2= H_{i_{1}}\cap \cdots \cap H_{i_{l}}$, hence $\prod_{j=1}^l t_{i_j}= \prod_{W_2\subset H_i}t_i\neq 1$.
\end{proof}

\bp\label{propLChi} 
Fix a dense edge $W\in E(\sA)$. For any $L \in \W(W)^{\circ}$, (\ref{length}) holds as equality. Moreover, \be 
\ell(Rj_{\ast} L[n])=1+\vert \chi(\pp(U_{W})) \vert.
\ee
\ep

\begin{proof}
If $L \in \W(W)^{\circ}$, then, with notation as in the proof of Theorem \ref{lower}, one has that $Q_{i}=0$ for any $i\neq w$,  and $Q_{w}= \oplus L^{W}[n-w]$ with $\dim H^{w}(U_{W},L_W)$ copies, according to Theorem \ref{arr} and the proof of Theorem \ref{lower}. With the same notation as in the latter, consider the inclusions: 
$$ U=U_{0} \hooklongrightarrow{j_{w,\cdots,1}} U_{w}  \hooklongrightarrow{j_{n,\cdots,w+1} } U_{n}=\CN.$$
We have that $$  R(j_{w-1,\cdots,1})_{\ast} (L[n])=(j_{w-1,\cdots,1})_{!\ast}  (L[n]).$$ Hence (\ref{si}) gives a short exact sequence of perverse sheaves
\be  \label{sw}
0 \to (j_{w,\cdots,1})_{!*} (L[n]) \to   R(j_{w,\cdots,1})_{\ast} (L[n]) \to  \oplus L^{W}[n-w]\to 0 
\ee
since $(j_{w,\cdots,1})_{!*}(L[n])=(j_w)_{!*}(j_{w-1,\cdots,1})_{!*}(L[n])$. Then (\ref{length}) holds as equality if we can show that $(j_{n,\cdots,w+1})_{!*}$ applied to this exact sequence is exact. If $w=n$, it is trivial. 

Assume $w<n$. Letting $i=w+1$ in (\ref{si}), since $Q_{w+1}=0$ one obtains that
$$R(j_{w+1,\ldots, 1})_*(L[n])= R(j_{w+1})_{*} R(j_{w,\cdots,1})_{\ast} (L[n])=$$
$$=(j_{w+1})_{!} R(j_{w,\cdots,1})_{\ast} (L[n])=(j_{w+1})_{!*} R(j_{w,\cdots,1})_{\ast} (L[n]).$$
Repeating this successively until $i=n$, one has
$$Rj_*(L[n])= R(j_{n,\cdots,w+1})_{*} R(j_{w,\cdots,1})_{\ast} (L[n])=$$
$$=(j_{n,\cdots,w+1})_{!} R(j_{w,\cdots,1})_{\ast} (L[n]) =(j_{n,\cdots,w+1})_{!*} R(j_{w,\cdots,1})_{\ast} (L[n]).$$
On the other hand,  $L \in \W(W)^{\circ}$ implies by Lemma \ref{lemNonRes} that  $$R(j_{n,\cdots,w+1})_{*}  (L^{W}[n-w])=(j_{n,\cdots,w+1})_{!}  (L^{W}[n-w])=(j_{n,\cdots,w+1})_{!*}  (L^{W}[n-w]).$$
Thus, $(j_{n,\cdots,w+1})_{!*}$ applied to the exact sequence (\ref{sw}) produces another exact sequence.



To prove the equality $\ell(Rj_{\ast} (L[n]))=1+\vert \chi(\pp(U^{W}))\vert$, we only need to show that $\dim H^{w}(U_{W}, L_W)=\vert \chi(\pp(U_{W}))\vert$. Note that  $\sA_W $ can be viewed as a central arrangement, hence $U_W= \pp(U_{W}) \times  \C^*$ (see \cite[Proposition 6.4.1]{D2}). For any $L \in \W(W),$ there  exists a unique rank one local system $\widetilde{L}_W$ on $\pp(U_{W})$ such that $p^{\ast}\widetilde{L}_W=L_W$, where $p: U_{W} \to \pp(U_{W})$ is the Hopf map (same as the first projection map) (see \cite[Proposition 6.4.3]{D2}). 
If $L \in \W(W)^{\circ}$, Lemma \ref{lemNonRes} implies that $\widetilde{L} _W$ on $\pp(U_W)$ satisfies the assumption in \cite[Theorem 6.4.18]{D2} from where it follows  
 that   \begin{center}
 $\dim H^{i}(\pp(U_{W}),\widetilde{L}_W )= \begin{cases}
\vert  \chi(\pp(U_{W}))\vert & \text{ if } i=w-1, \\
0 & \text{ if }i\neq w-1.
\end{cases}$
\end{center}
Then the K\"{u}nneth formula gives us that \begin{center}$\dim H^{i}(U_{W}, L_W )= \begin{cases}
\vert  \chi(\pp(U_{W}))\vert &\text{ if }  i=w-1, w, \\
0 & \text{otherwise.}
\end{cases}$\end{center}
This finishes the proof.
\end{proof}

\begin{proof}[{\bf Proof of Theorem \ref{thrmL2}}]

Clearly (b) implies (a). 

Assume that $\ell(Rj_{\ast} (L[n]))=2$. We know that $j_{!*}(L[n])$ is a  perverse subsheaf of length one. Hence, by Theorem \ref{arr}, there are two possibilities.

The first is when strict inequality happens in Theorem \ref{arr}. That is, $H^w(U_W,L_W)=0$ for all $W\neq\bC^n$. This is equivalent to $(j_{n,\ldots, i})_{!*}Q_i=0$ for all $i>0$, with notation as in the proof of Theorem \ref{lower}. Hence $Q_i=0$ for $i>0$, and $Rj_{\ast} (L[n])=j_{!*}(L[n])$. However, this has length one, which is a contradiction. So in fact this case does not occur.

The second case is when equality happens in Theorem \ref{arr}. That is,  there exists only one edge $W\neq \CN$ such that $ H^w(U_W,L_W)\neq 0$. Here $w=\codim W$. 

Firstly, $W$ has to be a dense edge. If not, there exists a finite collection of dense edges $\lbrace W_{i} \rbrace_{1\leq i \leq l}$ ($l\geq 2$) such that $U_{W}=U_{W_{1}} \times \cdots \times U_{W_{l}}$. By the K\"unneth formula, $$\dim H^w(U_W,L_W)=\prod_{i=1}^l \dim H^{w_i} (U_{W_i},L_{W_i} )\neq 0.$$ This contradicts the fact that  $ H^{w'} (U_{W'},L_{W'})=0$ for all edges $W'\neq W, \CN$, where $w'=\codim W'$. In particular, (a) implies (b).

Secondly, by the propagation property, $\sV^{w'}(U_{W'})=\sV(U_{W'})$ for all edges $W'\neq W, \CN$. Thus $L \notin \W(W')$, and hence $L\in \W(W)^\circ$.  Then the rest of the claim that (a) is equivalent to (c) follows now from  Lemma \ref{lemWo} and Proposition \ref{propLChi}.
\end{proof}


\section{Length in the plane arrangements case}\label{secPl}

We specialize now to dimension three and prove Theorem \ref{dim 3} and Proposition \ref{upper}. We keep the same notation as in the previous section, except $n=3$ now.

\subsection{Proof of Theorem \ref{dim 3}.}
\bd A subset $Z$ of the space of rank one local systems $M_B(U)$ is called {\it absolute $\bQ$-constructible subset}, if $Z$ is a Zariski constructible subset defined over $\bQ$ which is obtained from finitely many torsion-translated complex affine algebraic subtori of $M_B(U)$ via a finite sequence of taking union, intersection, and complement.
\ed 
\begin{proof}[Proof of Theorem \ref{dim 3}]
 Consider the ``good" subset of $M_B(U)$ consisting of local systems $L$ satisfying equality in (\ref{length}),
\be\label{eqGood}
\{L\in M_B(U)\mid \ell(Rj_*L[3])=\sum_W \dim H^w(U_W,L_W)\}.
\ee
By \cite{BW}, the cohomology loci 
$$
\{K\in M_B(U_W)\mid \dim H^w(U_W,K)=k_W\}
$$
of rank one local systems on $U_W$ are absolute $\bQ$-constructible subsets for any  $k_W\in\bN$. By \cite[Corollary 1.3]{BGLW}, the length loci
$$
\{L\in M_B(U)\mid \ell(Rj_*L[3])=k\}
$$
are also absolute $\bQ$-constructible for any $k\in\bN$. Moreover, the map $M_B(U)\ra M_B(U_W)$, $L\mapsto L_W$, is just a projection, hence the inverse image of a torsion-translated subtorus is also a torsion-translated subtorus. Thus this map pulls back absolute $\bQ$-constructible sets to absolute $\bQ$-constructible sets. By varying $k$ and the $k_W$ such that $k=\sum_W k_W$, it follows that the good set from (\ref{eqGood}) is absolute $\bQ$-constructible.

This reduces the proof of Theorem \ref{dim 3} to the torsion local system case. Indeed, if (\ref{length}) holds as equality for any torsion local system with $t_i\neq 1$ for all $i$, then smallest  absolute $\bQ$-constructible set containing all these torsion local systems is exactly the complement of $\bigcup_{i=1}^r\{ t_i=1\} $ in $M_B(U)$. 

Let $L_{\underline{t}}$ be a torsion rank one local system on $U$, that is, all coordinates of ${\underline{t}}\in(\bC^*)^r$ are roots of unity. 
Assume that all $t_i\neq 1$. Then $L_{\underline{t}}$ is trivialized on a finite Galois \'etale cover of $U$, hence $L_{\underline{t}}[3]$ is the complexification of a direct summand of a shifted higher rank local system underlying a pure Hodge module of weight 3, see \cite[Th\'eor\`eme 1.1, Lemme 3]{Sa2}. We shall consider the smallest such Hodge module. Namely, let $O_{\underline{t}}$ be the orbit of ${\underline{t}}$ in $(\bC^*)^r$ under the diagonal action of the Galois group $Aut(\bC/\bQ)$. Then for every $\underline{\al}\in O_{\underline{t}}$, $\al_i$ is a primitive root of unity of same order as $t_i$ for every $i$. Hence $O_{\underline{t}}$ is finite. The direct sum 
$$
\widehat{L}_\bC=\oplus_{{\underline{\al}}\in O_{\underline{t}}}L_{\underline{\al}}
$$
is a higher-rank local system on $U$, and it is an $Aut(\bC/\bQ)$-invariant $\bC$-point of the moduli space of local systems on $U$. Hence there exists a $\bQ$-local system $\widehat{L}$ such that $\widehat{L}_\bC=\widehat{L}\otimes_\bQ\bC$. { In particular, $\widehat{L}$ is a simple $\bQ$-local system.}  {{Moreover, $\widehat{L}[3]$ underlies a Hodge module of pure weight 3. }} 
{To see this, one can consider the finite cover $U^{\underline{t}}$ of $U$ associated to the following composition of surjective maps $$ \pi_1(U)\to H_1(U,\Z)\cong \Z^r \to \oplus_{i=1}^r \Z/(\ord  t_i) .$$ Here the first map is the Hurewicz morphism,
and the second map sends every factor $\Z$ to the finite group $\Z/(\ord t_i)$, where $\ord t_i$ is the order of $t_i$. The finite cover $U^{\underline{t}}$ can be taken as a algebraic variety such that the covering map $\pi: U^{\underline{t}} \to U$ becomes an algebraic map. Now $\bQ_{U^{\underline{t}}} [3]$ undelies a pure Hodge module of weight 3. Then so does $R\pi_* \bQ_{U^t} [3]$, since $\pi$ is a finite map. Note that  $\widehat{L}[3]$ is a direct summand  of $R\pi_* \bQ_{U^t} [3]$, so   $\widehat{L}[3]$ underlies a pure Hodge module of  weight 3.  In the rest of the proof, we use that all considered functors lift to the category of mixed Hodge modules, \cite[Theorem 0.1]{Sai}.


With the same notation as in the proof of Theorem \ref{lower}, let $U_{i}$ denote the union of all strata with codimension at most $i$, and consider sequence of maps: 
$$ U=U_{0} \overset{j_{1} }{\hookrightarrow} U_{1}  \overset{j_{2} }{\hookrightarrow} U_{2}   \overset{j_{3} }{\hookrightarrow} U_{3}=\bC^3.$$
We construct the following three short exact sequences of $\bQ$-perverse sheaves underlying similar exact sequences of mixed Hodge modules: 
\begin{align} 
0 \to (j_{1})_{!\ast} (\widehat{L}[3]) \to   R(j_{1})_{\ast} (\widehat{L}[3]) \to \widehat{Q}_{1} \to 0   \label{s31} \\
0 \to (j_{2})_{!\ast}  R(j_{1})_{\ast} (\widehat{L}[3]) \to   R(j_2 \circ j_1)_{\ast} (\widehat{L}[3]) \to \widehat{Q}_2 \to 0  \label{s32} \\
0 \to (j_{3})_{!\ast}  R(j_2 \circ j_1)_{\ast} (\widehat{L}[3]) \to   Rj_{\ast} (\widehat{L}[3]) \to \widehat{Q}_3 \to 0   \label{s33} 
\end{align}

Firstly, the explicit description of the $\bQ$-perverse sheaf $\widehat{Q}_i$ is as follows. Consider the $\bC$-perverse sheaf $Q_i$ constructed out of { the rank one $\bC$-local system $L_{\underline{t}}$} as in the proof of Theorem \ref{lower}. The perverse sheaf $Q_i$ is supported only on the disjoint union of the strata $U^W$ where $W$ are edges of codimension $i$. Over each $U^W$, $$(Q_i)_{|U^W}=((L_{\underline{t}}^W)^{\oplus h^i(U_W,(L_{\underline{t}})_W)})[3-i]$$ as shown in Proposition \ref{pushfoward}. Hhere we set $$h^i(U_W,(L_{\underline{t}})_W)): =\dim H^i(U_W,(L_{\underline{t}})_W)) .$$ Doing the same for all $L_{\underline{\al}}$ with $\underline{\al}\in (\bC^*)^r$ in the orbit $O_{\underline{t}}$, we obtain $\bC$-perverse sheaves $Q_{{\underline{\al}},i}$ with
$$
(Q_{{\underline{\al}},i})_{|U^W}=(((L_{\underline{\al}})^W)^{\oplus h^i(U_W,(L_{\underline{\al}})_W)})[3-i].
$$
We define
$$
\widehat{Q_{i,\bC}} = \oplus_{\underline{\al}\in O_{\underline{t}}} Q_{\underline{\al},i}.
$$
It is not immediately clear that this perverse sheaf is defined over $\bQ$. For that we have to note that 
\be\label{eqHal}h^i(U_W,(L_{\underline{\al}})_W)=h^i(U_W,(L_{\underline{t}})_W)\ee
for all $\underline{\al}\in O_{\underline{t}}$. This follows immediately from the fact that cohomology jump loci are defined over $\bQ$, since $(L_{\underline{\al}})_W$ must also be an $Aut(\bC/\bQ)$-conjugate of $L_W$ if $\underline{\al}$ is a conjugate of $\underline{t}$. Thus there exists a $\bQ$-perverse sheaf $\widehat{Q_i}$ with $\widehat{Q_i}\otimes_\bQ\bC=\widehat{Q_{i,\bC}}$. By functoriality, $\widehat{Q_i}=\widehat{Q}_i$ are the perverse sheaves appearing in the exact sequences above.

Now, since $t_i\neq 1$ for all $i$, for every $\underline{\al}\in O_{\underline{t}}$ one also has that $\al_i\neq 1$ for all $i$. Then $\widehat{Q}_1 \otimes_\bQ \C=\widehat{Q_1}\otimes_\bQ\bC=0$ by Remark \ref{remQ1}, and hence $\widehat{Q}_1=0$. 
This implies that $$(j_{1})_{!\ast} (\widehat{L}[3]) =   R(j_{1})_{\ast} (\widehat{L}[3]).$$
Then the claim follows if we can show that the functor $(j_3)_{!*}$ is exact on the short exact sequence (\ref{s32}). 

Now, since the intermediate extension functor preserves the weight,  $(j_{1})_{!\ast} (\widehat{L}[3]) =   R(j_{1})_{\ast} (\widehat{L}[3])$ is also a Hodge module of pure of weight 3. We will show that $\widehat{ Q}_2$ has pure weight 4.

This is  a local computation. Let $W\in E(\sA)$ be  a line. 
Due to Theorem \ref{lower} and Proposition \ref{pushfoward}, 
we get that $\widehat{ Q}_2\vert_{U^W} $ is a $\bQ$-local system on $U^W$ shifted by 1 with stalk isomorphic to $H^2(U_W, \widehat{ L}_W)$. Here $ \widehat{ L}_W$ is the $\bQ$-local system corresponding to $\widehat{L}\vert_{U_x}$ for any $x \in U^W$, due to the fact that $U_x$ is homotopy equivalent to $U_W$.

Assume that $W$ has multiplicity $m$, that is, there are exactly $m$ hyperplanes of $\sA$ containing $W$.  Then $U_W$ is just the central line arrangement with $m$ lines passing through the origin. Without loss of generality, we can assume that $W=H_1\cap \cdots \cap H_m$. 

If $\prod_{i=1}^m t_i \neq 1$, then equation (\ref{eqHal}) gives us that $H^2(U_W,(L_{\underline{\alpha}})_W)=0$ for any $\underline{\al} \in O_{\underline{t}}$, hence $\widehat{ Q}_2 \vert_{U^W}=0$. 

On the other hand, if $\prod_{i=1}^m t_i = 1$, then equation (\ref{eqHal}) implies that $\prod_{i=1}^m \alpha_i = 1$ for any $\underline{\al} \in O_{\underline{t}}$.
Notice that  $$U_W=\bC^* \times \pp(U_W),$$ where $\pp(U_W)=\mathbb{P}^1 - \{ m \text{  points}\}$ and the hyperplanes $\{ H_i \}_{1\leq i \leq m}$ are one-to-one corresponding to these $m$ points. Let $p$ be the projection from $U_W$ to $\pp(U_W)$. 
The product $\prod_{i=1}^r \alpha_i =1 $ implies that there exists a rank one $\C$-local system $\widetilde{L}_{\underline{\alpha},W}$ on $\pp(U_W)$ such that $L_{{\underline{\alpha}},W} =p^* \widetilde{L}_{{\underline{\alpha}},W}$. In fact, $\widetilde{L}_{{\underline{\alpha}}, W} $ is the rank one $\C$-local system on $\pp(U_W)$, which sends the meridian along the point corresponding to $H_i$ to $\alpha _i$.  Set $$\widetilde{L}_W= \oplus_{{\underline{\alpha}} \in O_{\underline{t}}} \widetilde{L}_{{\underline{\alpha}},W}.$$
$ \widetilde{L}_W $ is indeed a $\bQ$-local system, since $ \widehat{L}_W\vert_{U_W}= p^*  \widetilde{L}_W $.
Putting all together, we get the the following  isomorphisms of mixed Hodge structures:  $$ H^2(U_W, \widehat{L}_W) =H^2(U_W, p^*\widetilde{L}_W) =H^1(\bC^*,\bQ)\otimes H^1(\pp(U_W) ,  \widetilde{L}_W).$$
Here the second isomorphism follows from the K\"{u}nneth formula.
It is clear that tensoring with $H^1(\bC^*,\bQ)$ is the  Tate twist operation. 

We argue that $H^1(\pp(U_W) , \widetilde{L}_W)$ has pure weight 1. In fact, let $\tilde{j}$ denote the inclusion from $\pp(U_W)$ to $\mathbb{P}^1$. Since every $\widetilde{L}_{\alpha,W}$ has $\alpha_i\neq 1$ for all $i$, it is easy to check that $$R\tilde{j}_* ( \widetilde{L}_W[1])= \tilde{j} _{!*}(\widetilde{L}_W[1]),$$ hence $$H^1(\pp(U_W) ,  \widetilde{L}_W)=\mathrm{IH} ^1 (\mathbb{P}^1,  \widetilde{L}_W) .$$ 
Note that $\mathrm{IH} ^1 (\mathbb{P}^1,  \widetilde{L}_W)$ has pure weight 1, hence so does $H^1(\pp(U_W) ,  \widetilde{L}_W) $. Therefore, $ H^2(U_W, \widehat{ L}_W )$ has pure weight 3. Counting the shift by 1, we get that $\widehat{Q}_2\vert_{U^W}$ has pure weight 4.

With the above weights, \cite[Lemma 2.18]{BFNP} shows that the functor $(j_3)_{!*}$ is exact on the short exact sequence (\ref{s32}). Note that \cite[Lemma 2.18]{BFNP} is proved for the category of polarized mixed Hodge modules which are extendable to its closure in the analytic case. It is shown in \cite[page 313]{Sai} that this category is equivalent to the category of mixed Hodge modules in the algebraic case. Then the claim follows.
}
\end{proof}


\subsection{Deletion-restriction method.} To deal with the case where $t_i=1$ for some $i$, we recall first the deletion-restriction method in the general case of a hyperplane arrangement in $\CN$. Without loss of generalities, we assume that $t_1=1$. 
 The hyperplane $H_{1}$ gives a triple of arrangements $(\sA, \sA^{\prime}, \sA^{\prime \prime})$, where $\sA^{\prime} =\sA \setminus \lbrace H_{1} \rbrace $ is an arrangement in $\CN$ with one less hyperplane than $\sA$ and $\sA^{\prime \prime}$ is the arrangement $\sA^{H_{1}}$ we defined before. 
 
 Set $U^{\prime}= \CN \setminus \cup_{i=2}^{r} H_{i}$ and $U''=U'-U$. Consider the inclusions: $$U \overset{i}{ \hookrightarrow} U^{\prime} \overset{j^{\prime}}{ \hookrightarrow} \CN .$$
Then one can extend the local system $L$ on $U$ to $U'$, which is a rank one local system on $U^{\prime}$ denoted by $L'$, 
 such that $L=i^{-1} L^{\prime}$. Moreover, $i_{!*}(L[n])=L^{\prime}[n]$.  Using a similar approach as in the proof of Theorem \ref{lower}, one has the following short exact sequence of perverse sheaves: 
$$ 0\to i_{!*}(L[n]) \to Ri_{\ast} (L[n]) \to Q \to 0 $$
where $Q=L_{H_1}[n-1]$ follows from Proposition \ref{pushfoward}.  Set $L''=L_{H_1}$. We rewrite this short exact sequence as follows:
\be \label{triple}
 0\to L^{\prime}[n] \to Ri_{\ast} (L[n]) \to  L''[n-1] \to 0 \ee
Applying the functor $Rj^{\prime}_{\ast}$ to it, we got a new short exact sequence: 
$$ 0\to Rj^{\prime}_{\ast}(L^{\prime}[n]) \to Rj_{\ast} (L[n]) \to Rj^{\prime}_{\ast} ( L''[n-1]) \to 0 $$
which gives that $$ [Rj_{\ast} L[n]]= [Rj^{\prime}_{\ast}L^{\prime}[n]] +[Rj^{\prime}_{\ast}  L''[n-1]]$$
and 
\be  \label{equality}
\ell (Rj_{\ast} L[n])= \ell (Rj^{\prime}_{\ast}L^{\prime}[n]) + \ell(Rj^{\prime}_{\ast}  L''[n-1]).
\ee

\subsection{Length jump loci via cohomology jump loci.}
Now let us focus again on the dimension 3 case.
For simplicity, we assume that $\sA$ is a central essential hyperplane arrangement in $\bC^3$. Then $\sA'$ and $\sA''$ are both central arrangements. Let $\pp(U)$, $\pp(U'), \pp(U'') $ denote the projection of the complements $U,U',U''$, respectively. 
Assume that $\prod_{i=1}^r t_i=1$. Then the three local systems $(L, L',L'')$ induce three rank one local systems $(\tilde{L}, \tilde{L'}, \tilde{L''}) $ on $(\pp(U), \pp(U'), \pp(U'')) $, respectively. In particular, 
$$ H^3(U,L )=H^2(\pp(U), \tilde{L}), \text{  }  H^3(U',L' )=H^2(\pp(U'), \tilde{L'}), \text{  } H^2(U'',L'' )=H^1(\pp(U''), \tilde{L''})$$
Since we assume that $t_1=1$, we get a  short exact sequence similar to (\ref{triple}) for the new triples  $(\tilde{L}, \tilde{L'}, \tilde{L''}) $. 
 By taking  hypercohomology, one gets a long exact sequence:
 \be  \label{les}
 \begin{split}
 0\to H^1(\pp(U'),\tilde{ L'}) \to H^1(\pp(U),\tilde{ L}) \to & H^0(\pp(U''),\tilde{ L''}) \overset{\delta}{\to}  H^2(\pp(U'),\tilde{ L'}) 
\\  & \to H^2(\pp(U),\tilde{ L}) \to H^1(\pp(U),\tilde{ L''}) \to 0
 \end{split}
 \ee
 where $\delta$ is the boundary map. See \cite[p. 221-222]{D2} for more details about this kind of long exact sequences. 
\br The exactness of (\ref{les}) is proved in \cite[Theorem 4]{Co}. This method was used to show in \cite[Theorem 2]{Co} that, for any fixed integer $d\geq 2$, there exists hyperplane arrangement $\sA$ such that the jump locus $\sV^1(U)$ of its complement $U$  contains positive dimensional components which are translated by characters of order $d$.
\er 
 
We assume now that $t_1=1$ and $t_i\neq 1$ for any $2\leq i \leq r$. Then $L'$ satisfies 
 the assumption in Theorem \ref{dim 3}, hence (\ref{length}) holds as equality for  $ L'$. Now $L''$ is a rank one local system on $U''$, and $\sA''$ is a line arrangement. Then (\ref{length}) also holds as equality for $L''$ as mentioned in the introduction, see Remark \ref{positive results}(b). 
Therefore, the equality (\ref{equality}) gives a way to compute $\ell (Rj_* L[3])$ in this case.
 
 Let us compare (\ref{equality}) with the right hand side of (\ref{length}). Consider the corresponding difference for decomposition factors in the Grothendieck group $G(\bC^3)$.
 If we mod out the skyscraper sheaf $\bC_0$ (supported at the origin) part, the computation is same as the line arrangement case. Since (\ref{length}) holds as equality for the line arrangement case, the difference is only about the multiplicity of $\bC_0$,
  which is  as follows \footnote{Note that this gives a positive answer to Question \ref{question 3.6} in dimension 3 case.}:
 \be \label{diff}
 \dim  H^3(U',L')   + \dim H^2(U'',L'')- \dim H^3(U,L).\ee
If $\prod_{i=1}^r t_i \neq 1$, then all these three terms are 0, hence (\ref{length}) holds as equality. However, in our case $\prod_{i=1}^r t_i=1$, and this difference is same as 
 $$\dim H^2(\pp(U'),\tilde{ L'}) 
 +\dim H^1(\pp(U''),\tilde{ L''})- \dim H^2(\pp(U),\tilde{ L}) .$$
Then the long exact sequence (\ref{les}) shows that  this difference is 0 if and only if the boundary map $\delta $ is trivial. 
 So whether (\ref{length}) holds as equality for the length of $Rj_*L[3]$ depends on if the boundary map $\delta$ is trivial. 
 
 Note that $\dim H^0(\pp(U),\tilde{ L''}) $ is either 0 or 1, hence the difference is at most 1.
In fact, if $\tilde{ L''}$ is not the constant sheaf, then $H^0(U'',L'')=0$, hence $\delta=0$.
On the other hand, if $\tilde{ L''}$ is the constant sheaf, it is  a hard question to determine $\delta$.  
\begin{que} \label{boundary}
Is $\delta$ combinatorially  determined?
\end{que}
 This question has a positive answer if the jump loci of a hyperplane arrangement complement are combinatorially determined.

 Next we give an example where $\delta$ is non-trivial.
  
  \bex \label{counter} Consider the central hyperplane arrangement in $\bC^3$ defined by $z(x-z)^{m_1}(y-z)^{m_2}=0$ where $m_1\geq 2$ and $m_2\geq 2$. These were studied  in \cite{CDP}. 
  We order the hyperplanes as the factors. In particular, $H_1$ is the hyperplane defined by $z=0$.
  
   First, let us choose $t$ such that $\tilde{L''} $ is the constant sheaf. This can be done, if $t_1=1$, $\prod_{i=2}^{m_1+1} t_i=1$ and $\prod_{i=m_1+2}^{m_1+m-2+1} t_i=1$. 
   
 Secondly,  we choose  $t\notin \sV^1(\pp(U))$. The fundamental group of $\pp(U)$ is  a product of two free groups $F_{m_1}\times F_{m_2}$, see \cite[Corollary 1.7]{CDP}. Taking the projection of $H_1$ to be the hyperplane at infinity in $\pp(U)$, it is easy to show that $\sV^1(\pp(U))$ has two irreducible components: the component defined by $t_i=1 $ for any $2\leq i\leq m_1+1$, and the component defined by $t_i=1$ for any $m_1+2\leq i \leq m_1+m_2+1$.
 
 It is clear that we have a lot of choices of $t$ satisfying the above conditions. For such $t$, $H^1(\pp(U), \tilde{L})=0$ and $\dim H^0(\pp(U''),\tilde{L''})=1$, hence $\delta\neq 0$.   
  \eex
  
If the number of indices $i$ for which $t_i=1$ is $\ge 2$, one can repeat the above procedure and use (\ref{equality})  to compute $\ell (Rj_* L_{\underline{t}}[3])$. 
 
\begin{proof}[{\bf Proof of Proposition \ref{upper}}]
If $t_i\neq 1$ for all $i$, then (\ref{length}) shows that $\ell(Rj_* L_{\underline{t}}[3])$ can be computed from $\sV^2_j(U(\sA))$. 
If there is only one $i$ such that $t_i=1$, we assume that $t_1=1$ without loss of generalities. Note that $U''$ is a central line arrangement, hence $\dim H^2(U'',L'')$ is determined by combinational data. Then $\ell(Rj_* L_{\underline{t}}[3])$ can be computed from $ \sV^2_j(U(\sA'))$.  If there are more indices such that $t_i=1$, we repeat   the above procedure and use (\ref{equality}). Then it is clear that $\ell (Rj_* L_{\underline{t}}[3])$ can be computed by the jump loci $\sV_j^2(U(\sB))$ of the complement of all possible  
sub-arrangements $\sB$ of $\sA$ with all $j\geq 1$. The inequality follows by  induction and the fact that the difference (\ref{diff}) is at most 1.
\end{proof}
\br A general formula for  $\ell(Rj_* L[3])$ is  hard to obtain even for central essential hyperplane arrangement in $\bC^3$, since it is not clear if $\delta$ is combinatorially determined.
\er

\section{Intersection cohomology and characteristic cycles}\label{secIC}

As a by-product of the results for plane arrangements in $\bC^3$, we can address other invariants besides lengths in this case. In this section we prove Theorems \ref{cc} and \ref{PS}.

\subsection{General remarks.}
Let $P$ be a perverse sheaf on $\CN$ and let $\CC(P)$  denote the characteristic cycle of $P$. 
Note that, for any rank one local system $L$ on $U$, it is well known that $\CC(L[n])=\CC(\bC_U[n])$, see  \cite[Example 4.3.21(i)]{D2}.     Then the formula  \cite[(0.2.2)]{Gin} implies that $\CC(Rj_{\ast}L[n])  =\CC(Rj_{\ast}\C_U[n]),$ hence
\begin{align*}
\CC(Rj_{\ast}L[n]) & =\CC(Rj_{\ast}\C[n]) \\
& =\sum _{W\in E(\sA)}  \dim H^{w}(U_{W})\cdot \CC( \C_{W}[n-w]) \\
 &=  \sum _{W\in E(\sA)}   \dim H^{w}(U_{W})\cdot T_{W}^{\ast}\CN,
\end{align*}
where $T_{W}^{\ast}\CN$ is the conormal bundle over $W$ and the second equality follows from (\ref{constant}). 
Once (\ref{df}) and (\ref{length}) hold as equality, one can use them to give a formula for $\CC( j_{!*}(L[n]))$ by  induction on the dimension. 

We thank a referee for pointing out to us the following positive answer to a question we posed in the original version of this article:

\begin{prop}\label{upper bound} For a rank one local system $L$ on the complement $U$ of a hyperplane arrangement $\sA$ in $\bC^n$, one has 
$$
\ell (Rj_* L[n])\leq \ell  (Rj_* \bC_U[n]) = \sum _{W\in E(\sA)}   \dim H^{w}(U_{W}).
$$
Moreover, equality holds if and only if $L$ is the constant sheaf.
\end{prop}
\begin{proof} For a perverse sheaf  $P$ on $\bC^n$, one has a decomposition
$$ \CC(P)= \sum_i n_i\cdot \Lambda_i,$$
where each $\Lambda_i$ is an irreducible conic Lagrangian cycle and $n_i$ are positive integers, see \cite[Corollary 5.2.24]{D2}. Recall that the characteristic cycle functor factors through the Grothendieck group $G(\CN)$, see \cite[Section 4.3]{D2}. Thus $$\ell(P) \leq \sum_i n_i$$ and equality holds if and only if every decomposition factor of $P$ has exactly one irreducible conic Lagrangian cycle as its characteristics cycle.

This is an equality if $P=Rj_* \bC_U[n]$ by Remark \ref{positive results}(a), which also gives
$$
\ell(Rj_* \bC_U[n]) = \sum _{W\in E(\sA)}   \dim H^{w}(U_{W}).
$$
For a rank one local system $L$ on $U$, the equality $ \CC(Rj_{\ast}L[n])  =\CC(Rj_{\ast}\C[n])$ then implies
\be \label{ub}
\ell (Rj_* L[n])\leq \ell  (Rj_* \bC_U[n]).
\ee

{Moreover, if (\ref{ub}) is an equality, then for every decomposition factor $P$ of $Rj_* L[n]$,  $\CC(P)$ is exactly one irreducible conic Lagrangian cycle. We use now the same notation as in the proof of Theorem \ref{lower}. Thus  (\ref{eqMM}) implies that the decomposition factors of $R(j_1)_* L[n]$ are  exactly the non-zero $j^{-1}_{n,\ldots, 2}P$ with $P$ being all possible decomposition factors of $Rj_*L[n]$. In particular, the support of $Q_1$ is  the union of all hyperplanes in $\sA$ minus the codimension 2 edges. Therefore the restriction  of $R^1j_*L$ to each hyperplane minus the codimension 2 edges is non-zero. This can happen only if $L$ is the constant sheaf. }
\end{proof}

\subsection{Dimension 2 case.}
Now let us focus on $\C^{2}$. 
For an essential line arrangement $\sA$ in $\C^{2}$, we set $E(\sA)=\lbrace \bC^2, H_{1}, \cdots, H_{r}, p_{1}, \cdots, p_{l} \rbrace$, where $\lbrace p_{j} \rbrace$  are the singular points of $\sA$. Assume that the multiplicity associated to each singular point $p_{j}$ is $m_{j}$. 
 Then   $$\CC(Rj_{*}L[2])=\CC(Rj_*\bC_U[2])= T_{\C^{2}}^{\ast}\C^{2}+\sum_{i=1}^{r} T_{H_{i}}^{\ast}\C^{2} +\sum_{j=1}^{l} (m_{j}-1) \cdot T_{p_{j}}^{\ast}\C^{2}. $$
Moreover (\ref{df}) holds as equality in dimension 2 case, which gives us that:
  $$\CC(Rj_{*}L[2]))=\CC(j_{!*}(L[2]))+ \sum_{i\; : \; t_{i}=1} \CC(\IC(H_i, L_{H_{i}}))+ \sum _{j=1}^{l} \dim H^2(U_{p_j}, L)\cdot \CC(\C_{ p_{j}}),$$ 
 Note that if $L=L_{\underline{t}}$ with $t_i=1$, 
$$\CC(\IC(H_i, L_{H_{i}}))=T_{H_{i}}^{\ast}\C^{2}+ \sum T_{p_{j}}^{\ast}\C^{2},$$ where the last sum is over all the singular points $p_{j}\in H_{i}$ such that $ \prod_{p_{j}\in H_{k}}  t_{k} \neq 1$. Here the product is over all lines $H_k$ which contain $p_j$. 

Set  $$\gamma_{j}(t)=  \left\{ \begin{array}{ll}
0 & \text{if } t_{k}=1\text{ for all } k \text{ such that }  p_{j}\in H_{k},\\
m_j-1-\#\lbrace t_{k}=1 \mid p_{j}\in H_{k} \rbrace  &  \text{if }\prod_{p_{j}\in H_{k}}  t_{k} \neq 1, \\
1 & \text{otherwise}.\\
\end{array}\right. $$ 
Then we have the following combinational formula for the characteristic cycles of intersection complex $j_{!*}(L[2])$: 
$$
 \CC(j_{!*}(L_{\underline{t}}[2]))= T_{\C^{2}}^{\ast}\C^{2}+\sum_{i\; :\; t_{i}\neq 1} T_{H_{i}}^{\ast}\C^{2} +\sum_{j=1}^{l} \gamma_{j}(t) \cdot T_{p_{j}}^{\ast}\C^{2}. 
$$  

{ The Euler characteristic number is additive with respect to the distinguished triangles (see \cite[page 95]{D2}), hence the fact that (\ref{df}) holds as equality in dimension 2 case, gives us that $$\chi(U)=\chi(Rj_{*}L_{\underline{t}}[2])= \chi(\IC(\bC^2,L_{\underline{t}}))+ \sum_{i\; : \; t_{i}=1} \chi(\IC(H_i, L_{H_{i}}))+ \sum _{j=1}^{l} \dim H^2(U_{p_j}, L). $$
Note that when $t_i=1$, $$\chi (\IC(H_i, L_{H_{i}}))=-1+ \#\lbrace  p_{j}\in H_{i} \mid \prod_{p_{j}\in H_{k}}  t_{k} \neq 1 \},$$ where the product runs over all $H_k$ such that $p_j\in H_k$. On the other hand, $\chi(U)=1-r+\sum_j (m_j-1)$. Then we get that } $$ 
\chi(j_{!*}(L_{\underline{t}}[2]))=1 -\sum _{i\; :\; t_{i}\neq 1} 1 + \sum_{j=1}^{l} \gamma_{j}(t).
$$

Next we compute $\IH^*(\bC^2, L_{\underline{t}})$.  If $L_{\underline{t}}$ is the constant sheaf, then $\IH^*(\bC^2, L_{\underline{t}})=H^*(\bC^2).$ Now we assume that $L_{\underline{t}} \neq \bC_U$, hence $\IH^0(\bC^2, L_{\underline{t}})=0$. Let $\sB$ be the sub-arrangement of $\sA$ defined by all the hyperplanes $H_i$ with $t_i\neq 1$. Then $\sB\neq \emptyset$, since $L_{\underline{t}} \neq \bC_U$. If $t_i\neq 1$ for all $i$, $\sB=\sA$. 
Let $U(\sB)$ be the complement of the arrangement $\sB$. Let $i$ be the inclusion from $U$ to $U(\sB)$. Then there exists a unique rank one local system $L(\sB)$ on $U(\sB)$ such that $i_{!*}(L_{\underline{t}}[2])=L(\sB)[2]$, hence $\IH^*(\bC^2, L_{\underline{t}})=\IH^*(\bC^2, L(\sB)) $. 

Let $j^{\sB}$ denote the open inclusion from $U(\sB)$ to $\bC^2$. 
Consider the two short exact sequences as in the proof of Theorem \ref{lower} for $Rj^{\sB}_* L(\sB)[2]$. Consider the hypercohomology group long exact sequence associated to the second short exact sequence. Since the last term is a skyscraper sheaf, one gets the isomorphism 
$$H^{-1}(\bC^2, (j^{\sB}_2) _{!*}  R(j^{\sB}_1)_* (L(\sB)[2])) =H^1(U(\sB),L(\sB)) $$
 The first short exact sequence  gives an isomorphism $(j^{\sB}_1)_{!*} L(\sB)[1]= R(j^{\sB}_1)_*L(\sB)[1] $ due to the choice of $\sB$,  hence 
 $$\IH^1(\bC^2, L(\sB))=H^{-1}(\bC^2, (j^{\sB}) _{!*}   (L(\sB)[2])) = H^{-1}(\bC^2, (j^{\sB}_2) _{!*}  R(j^{\sB}_1)_* (L(\sB)[2])).$$
Putting all this together, one gets that $\IH^1(\bC^2, L_{\underline{t}})=H^1(U(\sB), L(\sB))$, and  $\IH^2(\bC^2, L_{\underline{t}})$ can be computed from $\IH^1(\bC^2, L_{\underline{t}})$ using the formula for $ \chi(j_{!*}(L_{\underline{t}}[2])).$
 
\subsection{Dimension 3 case.}
Due to Theorem \ref{dim 3}, one can  compute $\CC( j_{!*}(L_{\underline{t}}[3]))$ for dimension 3 case, when $t_i \neq 1$ for all $i$.  However the formula would be quite complicated, as we can already see from the  dimension 2 case in previous subsection. 
For simplicity, we only give the formula for central essential hyperplane arrangements in $\bC^3$ and local systems $L$ with same monodromy around each hyperplane. 

Let $\sA$ be a central essential hyperplane arrangement in $\bC^3$. Let $f_i$  be a degree one polynomial  defining  $H_i$. 
The Milnor fiber $F$ associated to $\sA$ is defined by $\prod_{i=1}^r f_i =1 $ in $\bC^3$. The monodromy action on $F$ is given by $h:F\to F, $ $h(x)=\exp(2 \pi i /r)\cdot x$. Then the monodromy action on $H^*(F)$ is semi-simple and has order $r$. Let $H^*(F)_s$ denote the eigenspace of $H^*(F)$ with eigenvalue $s$.


\begin{proof}[{\bf Proof of Theorem \ref{cc}}]
Note that for $s\neq 1$ and any 2-dimensional edge $H_i$ in $\bC^3$, $H^1(U_x, L)=0$ for any generic point $x$ in $H_i$ and all $i$. Theorem \ref{dim 3} shows that (\ref{df}) holds as equality for $L_s$ when $s\neq 1$, which  gives us that 
$$\CC(Rj_{\ast}L[3])= \CC(j_{!\ast}(L[3]))+$$
$$+\sum_{j=1}^l \dim H^2(U_{\Lambda_j}, L_{\Lambda_j}) \cdot\CC( \IC(\Lambda_j, L^{\Lambda_j}))+\dim H^3(U,L) \cdot \CC(\bC_0).$$
Note that $U_{\Lambda_j} =\Lambda_j-0= \bC^*$ for all $j$. The representation of the rank one local system $L^{\Lambda_j}$ sends the only generator of the fundamental group to $s^{r-m_j}$. So 
$$\CC(\IC(\Lambda_j, L^{\Lambda_j}))=\left\{ \begin{array}{ll}
T^*_{\Lambda_j} \bC^3,  &  \text{if }s^{r-m_j}=1, \\
T^*_{\Lambda_j} \bC^3+T^*_0 \bC^3, & \text{otherwise}. \\ 
\end{array}\right.
$$ 
Since $\sA_{\Lambda_j}$ is a central line arrangement with $m_j$ lines and $s\neq 1$, 
$$\dim H^2(U_{\Lambda_j}, L_{\Lambda_j})=\left\{ \begin{array}{ll}
m_j-2,  &   \text{if }s^{m_j}=1, \\
0, & \text{otherwise}. \\ 
\end{array}\right.
$$ 
Since the monodromy action on $H^*(F)$ is semi-simple, it follows from \cite[Corollary 6.4.9]{D2} that $\dim H^3(U, L_s)= \dim H^2(F)_s$ and $\dim H^1(U, L_s)= \dim H^1(F)_s $ when $s\neq 1$. Recall that
$$
\CC(Rj_{\ast}L[3])  =\CC(Rj_{\ast}(\C_U[3]))=$$
$$ =T_{\C^{3}}^{\ast}\C^{3}+\sum_{i=1}^{r} T_{H_{i}}^{\ast}\C^{3} +\sum_{j=1}^{l} (m_{j}-1) \cdot T_{\Lambda_{j}}^{\ast}\C^{3}+ (\sum_{j=1}^l  (m_{j}-1) -r+1) \cdot T_{0}^{\ast}\C^{3}
$$
Putting all this  together, we get the desired formula for  $\CC(j_{!*}(L_s[3]))$. { Since we are dealing with central hyperplane arrangements,} the formula for $\chi(\bC^3, j_{!*}(L_s[3]))$ follows from the local index formula for characteristic cycles, e.g.  \cite[Theorem 4.3.25]{D2}. 

Next we compute the intersection cohomology $\IH^*(\bC^3, L_s)$. Since $s\neq 1$, $\IH^0(\bC^3, L_s)=0$. Now $\sA$ is  a central arrangement, hence $\IH^3(\bC^3, L_s)= H^0 (j_{!*} L_s[3])_0=0$. Here $H^0 (j_{!*} L_s[3])_0$ is 0-th cohomology  of the stalk of $ j_{!*} L_s[3]$ at the origin, and it is 0 due to the basic properties of intermediate extension. 

Now we recall the three short exact sequences (\ref{s31}), (\ref{s32}, (\ref{s33}) from the proof of Theorem \ref{dim 3}. The proof of Theorem \ref{dim 3}  showed that the functor $(j_3)_{!*}$ is exact on the short exact sequence (\ref{s32}) when $s\neq 1$. Forgetting the conjugated part, we have a new short exact sequence:
$$0 \to j_{!\ast}   (L_s[3]) \to   (j_3)_{!*} R(j_2 \circ j_1)_{\ast} (L_s[3]) \to (j_3)_{!*} Q_2 \to 0  $$
Consider the hypercohomology long exact sequence associated to this short exact sequence. Since the support of $(j_3)_{!*} Q_2 $ has dimension 1 and $(j_3)_{!*} Q_2 $ is perverse,  one gets that $H^{i}(\bC^3, (j_3)_{!*} Q_2)=0$ for $i<-1$ (\cite[Proposition 5.2.20]{D2}), which implies the following isomorphism $$\IH^1(\bC^3, L_s)=H^{-2}(\bC^3, (j_3)_{!*} R(j_2 \circ j_1)_{\ast} (L_s[3])).$$

Consider the  hypercohomology long exact sequence associated to the short exact sequence (\ref{s33}). Since $Q_3$ is the skyscraper sheaf, one gets that 
 $$H^{-2}(\bC^3, (j_3)_{!*} R(j_2 \circ j_1)_{\ast} (L_s[3]))=H^1(U,L_s).$$
 Hence $\dim \IH^1(\bC^3, L_s)=\dim H^1(U, L_s)=\dim H^1(F)_s$. 
 Note that $$-\chi(\bC^3, j_{!*}(L_s[3]))= \sum_{i=0}^3 (-1)^i \dim \IH^i(\bC^3, L_s)= \dim \IH^2(\bC^3, L_s)-\dim \IH^1(\bC^3, L_s) .$$
 Then the formula for $ \dim \IH^2(\bC^3, L_s)$ follows.  
\end{proof}

\begin{proof}[{\bf Proof of Theorem \ref{PS}}]
 
Using Theorem \ref{cc}, all that is left is to compute $\dim H^2(F)_s$. Papadima-Suciu \cite[Theorem 1.2]{PS} showed that $\Delta^1(t)=(t-1)^{r-1}(t^2+t+1)^{\beta_3(\sA)}$ . Note that if $3\nmid r,$ then $\beta_3(\sA)=0$. One has
$$\dim H^1(F)_s= \left\{ \begin{array}{ll}
\beta_3(\sA), &   s^3=1, s^r=1.\\
0, & \text{else}, \\ 
\end{array}\right. $$ Recall that $$\dfrac{\Delta^0(t) \Delta^2(t)}{\Delta^1(t)}= (t^r-1)^{\chi(F)/r}.$$
It is easy to check that $ \chi(F)/r= \binom{r-2}{2}-n_3(\sA),$ 
hence $$\Delta^2(t)=(t-1)^{r-2} (t^r-1)^{\binom{r-2}{2}-n_3(\sA)}(t^2+t+1)^{\beta_3(\sA)}.$$
{ Note that $\dim H^3(U, L_s)=\dim H^2(F)_s$ for $s\neq 1$.} 
Then the claims follow by direct computations.
\end{proof}



\begin{thebibliography}{BGLW17}
\bibitem[AB10]{AB1} T. Abebaw, R. B\o{}gvad, {\it Decomposition of D-modules over a hyperplane arrangement in the plane}. Ark. Mat. 48 (2010), no. 2, 211-229.

\bibitem[AB12]{AB2} T. Abebaw, R. B\o{}gvad, {\it Decomposition factors of D-modules on hyperplane configurations in general position.} Proc. Amer. Math. Soc. 140 (2012), no. 8, 2699--2711. 

\bibitem[BG18]{BG} R. B\o{}gvad, I. Gon\c{c}alves, {\it Length and decomposition of the cohomology of the complement to a hyperplane arrangement.} arXiv:1703.07662. To appear in Proc. Amer. Math. Soc.


\bibitem[BBD82]{BBD} A. Bei­linson, J. Bernstein, P. Deligne, {\it Faisceaux pervers.}  Ast\'{e}risque, 100, Soc. Math. France, Paris, 1982. 

\bibitem[Bib16]{Bib} C. Bibby, {\it Cohomology of abelian arrangements.} Proc. Amer. Math. Soc. 144 (2016), pp. 3093-3104.


\bibitem[BFNP09]{BFNP}
P. Brosnan, H. Fang, Z. Nie, G. Pearlstein, {\it Singularities of admissible normal functions.
With an appendix by Najmuddin Fakhruddin.} Invent. Math. 177 (2009), no. 3, 599-629. 

\bibitem[Bud15]{B-ls} N. Budur, {\it Bernstein-Sato ideals and local systems.} Ann. Inst. Fourier 65 no. 2, 549-603 (2015).


\bibitem[BGLW17]{BGLW} N. Budur, P. Gatti, Y. Liu, B. Wang, {\it On the length of perverse sheaves and $\sD$-modules.} arXiv:1709.00876. To appear in Bull. Math. Soc. Sci. Math. Roumanie.

\bibitem[BLSW17]{BLSW} N. Budur Y. Liu, L. Saumell, B. Wang, {\it Cohomology support loci of local systems.}   Michigan Math. J. 66 (2017), no. 2, 295-307.





\bibitem[BS10]{BS} N. Budur, M. Saito, {\it Jumping coefficients and spectrum of a hyperplane arrangement.} Math. Ann. 347 (2010), no. 3, 545-579. 



\bibitem[BW15a]{BW} N. Budur, B. Wang, {\it Cohomology jump loci of quasi-projective varieties.}  Ann. Sci. \'Ec. Norm. Sup\'er. (4) 48 (2015), no. 1, 227--236. 




\bibitem[CDP05]{CDP}  A. D. R. Choudary, A. Dimca, \c{S}. Papadima, {\it  Some analogs of Zariski's theorem on nodal line arrangements.}  Algebr. Geom. Topol. 5 (2005), 691-711. 

\bibitem[Coh02]{Co} D. C. Cohen, {\it Triples of arrangements and local systems.}  Proc. Amer. Math. Soc. 130 (2002), no. 10, 3025-3031.



\bibitem[DM09]{DM} M. A. de Cataldo, L. Migliorini,  {\it The decomposition theorem, perverse sheaves and the topology of algebraic maps.} Bull. Amer. Math. Soc. (N.S.) 46 (2009), no. 4, 535--633.


\bibitem[DSY15]{DSY} G. Denham, A. Suciu, S. Yuzvinsky, {\it Abelian duality and propagation of resonance}, Selecta Math. 23 (2017), no. 4, 2331-2367.

\bibitem[Dim92]{D1} A. Dimca, {\it Singularities and topology of hypersurfaces},
Universitext, Springer-Verlag, New York 1992.


\bibitem[Dim04]{D2} A. Dimca, {\it Sheaves in topology.} Universitext. Springer-Verlag, Berlin, 2004. xvi+236 pp.

\bibitem[Dim17]{D3} A. Dimca, {\it Hyperplane arrangements. An introduction.} Universitext. Springer, Cham, 2017. xii+200 pp.



\bibitem[Dur83]{Dur} A. Durfee,
{\it Neighborhoods of algebraic sets},
Trans. Amer. Math. Soc.  276 (1983), no. 2, 517--530.



\bibitem[Gin86]{Gin} V. Ginsburg, {\it Characteristic varieties and vanishing cycles.} Invent. Math. 84 (1986), no. 2, 327-402. 




\bibitem[Lo93]{Lo}  E. Looijenga, {\it Cohomology of $M_3$ and $M^1_3$,} Mapping class groups and moduli spaces of Riemann surfaces, (G\"ottingen, 1991/Seattle, WA, 1991), Comtemp. Math. 150, 1993,
205-228.


\bibitem[NT05]{NT} P. Nang, K. Takeuchi, {\it Characteristic cycles of perverse sheaves and Milnor fibers}.  
 Math. Z. 249 (2005), no. 3, 493-511.

\bibitem[Oak15]{Oa} T. Oaku, {\it Length and multiplicity of the local cohomology with support in a hyperplane arrangement.}  arXiv:1509.01813.

\bibitem[OT92]{OT} P. Orlik, H. Terao, {\it Arrangements of hyperplanes.}  Springer-Verlag, Berlin, 1992. xviii+325 pp.

\bibitem[PS17]{PS} \c{S}. Papadima, A. Suciu, {\it  The Milnor fibration of a hyperplane arrangement: from modular resonance to algebraic monodromy.}  Proc. Lond. Math. Soc. (3) 114 (2017), no. 6, 961-1004. 

\bibitem[Pe17]{Pe} D. Petersen, {\it A spectral sequence for stratified spaces and configuration spaces of points.}
Geom. Topol. 21 (2017), no. 4, 2527-2555.




\bibitem[Sa88]{Sa2}{M. Saito}, {\it Modules de Hodge polarisables}, Publ. Res. Inst. Math. Sci. 24, no. 6 (1988), 849-995.

\bibitem[Sai90]{Sai} M. Saito, {\it Mixed Hodge modules.} Publ. Res. Inst. Math. Sci. 26 (1990), 221-333.

\bibitem[STV95]{STV} V. Schechtman, H. Terao, A. Varchenko, {\it Local systems over complements of hyperplanes and the Kac-Kazhdan conditions for singular vectors.} J. Pure Appl. Algebra 100 (1995), 93-
102.


 \end{thebibliography}
\end{document}